\documentclass{amsart}

\usepackage{amssymb,upref}
\usepackage[mathcal]{euscript}
\usepackage[cmtip,all]{xy}

\newcommand{\bydef}{:=}

\newcommand{\bg}{{\overline{g}}}

\newcommand{\be}{{\overline{e}}}
\newcommand{\bG}{{\overline{G}}}

\newcommand{\sks}{\mathcal{K}}
\newcommand{\sg}{\mathrm{Sym}}

\newcommand{\cA}{\mathcal{A}}

\newcommand{\cB}{{\mathcal B}}

\newcommand{\cC}{\mathcal{C}}
\newcommand{\cD}{{\mathcal D}}

\newcommand{\cL}{\mathcal{L}}

\newcommand{\cM}{{\mathcal M}}
\newcommand{\cR}{\mathcal{R}}

\newcommand{\cS}{\mathcal{S}}

\newcommand{\W}{W} 

\newcommand{\id}{{\mathrm{id}}} 

\newcommand{\ul}[1]{\underline{#1}}
\newcommand{\wh}[1]{\widehat{#1}}
\newcommand{\wt}[1]{\widetilde{#1}}
\newcommand{\wb}[1]{\overline{#1}}
\newcommand{\vphi}{\varphi}
\newcommand{\veps}{\varepsilon}

\newcommand{\diag}{\mathrm{diag}}
\newcommand{\sgn}{\mathrm{sgn}}
\DeclareMathOperator*{\ot}{\otimes}

\newcommand{\ZZ}{\mathbb{Z}}

\newcommand{\FF}{\mathbb{F}}

\newcommand{\chr}[1]{\mathrm{char}\,#1}

\DeclareMathOperator{\End}{\mathrm{End}}

\DeclareMathOperator{\Aut}{\mathrm{Aut}}
\DeclareMathOperator{\Stab}{\mathrm{Stab}}
\DeclareMathOperator{\Diag}{\mathrm{Diag}}
\DeclareMathOperator{\Diags}{\mathbf{Diag}}
\DeclareMathOperator{\AAut}{\mathbf{Aut}}
\DeclareMathOperator{\Der}{\mathrm{Der}}

\DeclareMathOperator{\supp}{\mathrm{Supp}\,}

\newcommand{\brac}[1]{{#1}^{(-)}}

\newcommand{\Sl}{\mathfrak{sl}}

\newcommand{\Psl}{\mathfrak{psl}}
\newcommand{\So}{\mathfrak{so}}
\newcommand{\Sp}{\mathfrak{sp}}

\newcommand{\GL}{\mathrm{GL}}

\newcommand{\SP}{\mathrm{Sp}}

\newcommand{\M}{\Gamma_\cM}
\newcommand{\AI}{\Gamma_A^\mathrm{(I)}}
\newcommand{\AII}{\Gamma_A^\mathrm{(II)}}
\newcommand{\B}{\Gamma_B}
\newcommand{\C}{\Gamma_C}
\newcommand{\D}{\Gamma_D}

\newtheorem{theorem}{Theorem}[section]
\newtheorem{proposition}[theorem]{Proposition}
\newtheorem{lemma}[theorem]{Lemma}

\theoremstyle{definition}
\newtheorem{df}[theorem]{Definition}

\theoremstyle{remark}

\begin{document}

\title[Weyl groups of fine gradings on simple Lie algebras]{Weyl groups of fine gradings\\ on simple Lie algebras of types $A$, $B$, $C$ and $D$}

\author[Alberto Elduque]{Alberto Elduque$^{\star}$}
\thanks{$^{\star}$ Supported by the Spanish Ministerio de Educaci\'{o}n y Ciencia and
FEDER (MTM 2010-18370-C04-02) and by the Diputaci\'on General de Arag\'on (Grupo de Investigaci\'on de \'Algebra)}
\address{Departamento de Matem\'aticas e Instituto Universitario de Matem\'aticas y Aplicaciones,
Universidad de Zaragoza, 50009 Zaragoza, Spain}
\email{elduque@unizar.es}

\author[Mikhail Kochetov]{Mikhail Kochetov$^{\star\star}$}
\thanks{$^{\star\star}$Supported by the Natural Sciences and Engineering Research Council (NSERC) of Canada, Discovery Grant \# 341792-07.}
\address{Department of Mathematics and Statistics,
Memorial University of Newfoundland, St. John's, NL, A1C5S7, Canada}
\email{mikhail@mun.ca}




\begin{abstract}
Given a grading $\Gamma:\cL=\bigoplus_{g\in G}\cL_g$ on a nonassociative algebra $\cL$ by an abelian group $G$, we have two subgroups of $\Aut(\cL)$: the automorphisms that stabilize each component $\cL_g$ (as a subspace) and the automorphisms that permute the components. By the Weyl group of $\Gamma$ we mean the quotient of the latter subgroup by the former. In the case of a Cartan decomposition of a semisimple complex Lie algebra, this is the automorphism group of the root system, i.e., the so-called extended Weyl group. A grading is called fine if it cannot be refined. We compute the Weyl groups of all fine gradings on simple Lie algebras of types $A$, $B$, $C$ and $D$ (except $D_4$) over an algebraically closed field of characteristic different from $2$.
\end{abstract}

\maketitle


\section{Introduction}

In \cite{EK_Weyl}, we computed the Weyl groups of all fine gradings on matrix algebras, the Cayley algebra $\cC$ and the Albert algebra $\cA$ over an algebraically closed field $\FF$ ($\chr{\FF}\ne 2$ in the case of the Albert algebra). It is well known that $\Der(\cC)$ is a simple Lie algebra of type $G_2$ ($\chr{\FF}\ne 2,3$) and $\Der(\cA)$ is a simple Lie algebra of type $F_4$  ($\chr{\FF}\ne 2$). Since the automorphism group schemes of $\cC$ and $\Der(\cC)$, respectively $\cA$ and $\Der(\cA)$, are isomorphic, the classification of fine gradings on $\Der(\cC)$, respectively $\Der(\cA)$, is the same as that on $\cC$, respectively $\cA$ \cite{EK_Albert} and, moreover, the Weyl groups of the corresponding fine gradings are isomorphic. The situation with fine gradings on the simple Lie algebras belonging to series $A$, $B$, $C$ and $D$ is more complicated, because the fine gradings on matrix algebras yield only a part of the fine gradings on the simple Lie algebras of series $A$ (so-called Type I gradings). In order to obtain the fine gradings for series $B$, $C$ and $D$ and the remaining (Type II) fine gradings for series $A$, one has to consider fine $\vphi$-gradings on matrix algebras, which were introduced and classified in \cite{E10}.

The purpose of this paper is to compute the Weyl groups of all fine gradings on the simple Lie algebras of series $A$, $B$, $C$ and $D$, with the sole exception of type $D_4$ (which differs from the other types due to the triality phenomenon), over an algebraically closed field $\FF$ of characteristic different from $2$. To achieve this, we first determine the automorphisms of each fine $\vphi$-grading on the matrix algebra $\cR=M_n(\FF)$, $n\ge 3$, and then use the transfer technique of \cite{BK10} to obtain the Weyl group of the corresponding fine grading on the simple Lie algebra $\cL=[\cR,\cR]/(Z(\cR)\cap [\cR,\cR])$ or $\sks(\cR,\vphi)$, where in the second case $\vphi$ is an involution on $\cR$ and $\sks(\cR,\vphi)$ stands for the set of skew-symmetric elements with respect to $\vphi$.

We adopt the terminology and notation of \cite{EK_Weyl}, which is recalled in Section \ref{se:generalities} for convenience of the reader. In Section \ref{se:matrix}, we restate the classification of fine $\vphi$-gradings on matrix algebras \cite{E10} in more explicit terms and determine the relevant automorphism groups of each fine $\vphi$-grading (Theorem \ref{groups_fine_phi_grad_matrix}). In Section \ref{se:A}, we deal with the simple Lie algebras of series $A$ (Theorems \ref{th:Weyl_AI} and \ref{th:Weyl_AII}) and, in Section \ref{se:BCD}, with those of series $B$, $C$ and $D$ (Theorems \ref{th:Weyl_B} and \ref{th:Weyl_CD}).

\section{Generalities on gradings}\label{se:generalities}

Let $\cA$ be an algebra (not necessarily associative) over a field $\FF$ and let $G$ be a group (written multiplicatively).

\begin{df}\label{df:G_graded_alg}
A {\em $G$-grading} on $\cA$ is a vector space decomposition
\[
\Gamma:\;\cA=\bigoplus_{g\in G} \cA_g
\]
such that
\[
\cA_g \cA_h\subset \cA_{gh}\quad\mbox{for all}\quad g,h\in G.
\]
If such a decomposition is fixed, we will refer to $\cA$ as a {\em $G$-graded algebra}.
The nonzero elements $a\in\cA_g$ are said to be {\em homogeneous of degree $g$}; we will write $\deg a=g$.
The {\em support} of $\Gamma$ is the set $\supp\Gamma\bydef\{g\in G\;|\;\cA_g\neq 0\}$.
\end{df}

There are two natural ways to define equivalence relation on graded algebras. We will use the term ``isomorphism'' for the case when the grading group is a part of definition  and ``equivalence'' for the case when the grading group plays a secondary role. Let
\[
\Gamma:\; \cA=\bigoplus_{g\in G} \cA_g\mbox{ and }\Gamma':\;\cB=\bigoplus_{h\in H} \cB_h
\]
be two gradings on algebras, with supports $S$ and $T$, respectively.

\begin{df}\label{df:equ_grad}
We say that $\Gamma$ and $\Gamma'$ are {\em equivalent} if there exists an isomorphism of algebras $\psi\colon\cA\to\cB$ and a bijection $\alpha\colon S\to T$ such that $\psi(\cA_s)=\cB_{\alpha(s)}$ for all $s\in S$. Any such $\psi$ will be called an {\em equivalence} of $\Gamma$ and $\Gamma'$ (or of $\cA$ and $\cB$ if the gradings are clear from the context).
\end{df}

The algebras graded by a fixed group $G$ form a category where the morphisms are the {\em homomorphisms of $G$-graded algebras}, i.e., algebra homomorphisms $\psi\colon\cA\to\cB$ such that $\psi(\cA_g)\subset\cB_g$ for all $g\in G$.

\begin{df}\label{df:iso_grad}
In the case $G=H$, we say that $\Gamma$ and $\Gamma'$ are {\em isomorphic} if $\cA$ and $\cB$ are isomorphic as $G$-graded algebras, i.e., there exists an isomorphism of algebras $\psi\colon\cA\to\cB$ such that $\psi(\cA_g)=\cB_{g}$ for all $g\in G$.
\end{df}

It is known that if $\Gamma$ is a grading on a simple Lie algebra, then $\supp\Gamma$ generates an abelian group (see e.g. \cite[Proposition 3.3]{Ksur}). From now on, we will assume that our grading groups are {\em abelian}. Given a group grading $\Gamma$ on an algebra $\cA$, there are many groups $G$ such that $\Gamma$ can be realized as a $G$-grading, but there is one distinguished group among them \cite{PZ}.

\begin{df}\label{df:univ_group}
Suppose that $\Gamma$ admits a realization as a $G_0$-grading for some group $G_0$. We will say that $G_0$ is a {\em universal group of $\Gamma$} if, for any other realization of $\Gamma$ as a $G$-grading, there exists a unique homomorphism $G_0\to G$ that restricts to identity on $\supp\Gamma$.
\end{df}

One shows that the universal group, which we denote by $U(\Gamma)$, exists and depends only on the equivalence class of $\Gamma$. Indeed, $U(\Gamma)$ is generated by $S=\supp\Gamma$ with defining relations $s_1s_2=s_3$ whenever $0\ne\cA_{s_1}\cA_{s_2}\subset\cA_{s_3}$ ($s_i\in S$).

As in \cite{PZ}, we associate to $\Gamma$ three subgroups of the automorphism group $\Aut(\cA)$ as follows.

\begin{df}\label{df:aut_diag}
The {\em automorphism group of $\Gamma$}, denoted $\Aut(\Gamma)$, consists of all automorphisms of $\cA$ that permute the components of $\Gamma$. Each $\psi\in\Aut(\Gamma)$ determines a self-bijection $\alpha=\alpha(\psi)$ of the support $S$ such that $\psi(\cA_s)=\cA_{\alpha(s)}$ for all $s\in S$.
The {\em stabilizer of $\Gamma$}, denoted $\Stab(\Gamma)$, is the kernel of the homomorphism $\Aut(\Gamma)\to\sg(S)$ given by $\psi\mapsto\alpha(\psi)$. Finally, the {\em diagonal group of $\Gamma$}, denoted $\Diag(\Gamma)$, is the subgroup of the stabilizer consisting of all automorphisms $\psi$ such that the restriction of $\psi$ to any homogeneous component of $\Gamma$ is the multiplication by a (nonzero) scalar.
\end{df}

Thus $\Aut(\Gamma)$ is the group of self-equivalences of the graded algebra $\cA$ and $\Stab(\Gamma)$ is the group of automorphisms of the graded algebra $\cA$. Also, $\Diag(\Gamma)$ is isomorphic to the group of characters of $U(\Gamma)$ via the usual action of characters on $\cA$: if $\Gamma$ is a $G$-grading (in particular, we may take $G=U(\Gamma)$), then any character $\chi\in\wh{G}$ acts as an automorphism of $\cA$ by setting $\chi*a=\chi(g)a$ for all $a\in\cA_g$ and $g\in G$. If $\dim\cA<\infty$, then $\Diag(\Gamma)$ is a diagonalizable algebraic group (quasitorus). If, in addition, $\FF$ is algebraically closed and $\chr{\FF}=0$, then $\Gamma$ is the eigenspace decomposition of $\cA$ relative to $\Diag(\Gamma)$ (see e.g. \cite{Ksur}), the group $\Stab(\Gamma)$ is the centralizer of $\Diag(\Gamma)$, and $\Aut(\Gamma)$ is its normalizer. If we want to work over an arbitrary field $\FF$, we can define the subgroupscheme $\Diags(\Gamma)$ of the automorphism group scheme $\AAut(\cA)$ as follows:
\[
\Diags(\Gamma)(\cS)\bydef\{f\in\Aut_\cS(\cA\ot\cS)\;|\;f|_{\cA_g\ot\cS}\in\cS^\times\id_{\cA_g\ot\cS}\mbox{ for all }g\in G\}
\]
for any unital commutative associative algebra $\cS$ over $\FF$. Thus $\Diag(\Gamma)$ is the group of $\FF$-points of $\Diags(\Gamma)$. One checks that $\Diags(\Gamma)=U(\Gamma)^D$, the Cartier dual of $U(\Gamma)$, also $\Stab(\Gamma)$ is the centralizer of $\Diags(\Gamma)$ and $\Aut(\Gamma)$ is its normalizer with respect to the action of $\Aut(\cA)$ on $\AAut(\cA)$ by conjugation (see e.g. \cite[\S 2.2]{EK_Albert}).

\begin{df}
The quotient group $\Aut(\Gamma)/\Stab(\Gamma)$, which is a subgroup of $\sg(S)$, will be called the {\em Weyl group of $\Gamma$} and denoted by $\W(\Gamma)$.
\end{df}

It follows from the universal property of $U(\Gamma)$ that, for any $\psi\in\Aut(\Gamma)$, the bijection $\alpha(\psi)\colon\supp\Gamma\to\supp\Gamma$ extends to a unique automorphism of $U(\Gamma)$. This gives an action of $\Aut(\Gamma)$ by automorphisms of $U(\Gamma)$. Since the kernel of this action is $\Stab(\Gamma)$, we may regard $\W(\Gamma)=\Aut(\Gamma)/\Stab(\Gamma)$ as a subgroup of $\Aut(U(\Gamma))$. Given a $G$-grading $\Gamma:\;\cA=\bigoplus_{g\in G}\cA_g$ and a group homomorphism $\alpha\colon G\to H$, we obtain the induced $H$-grading ${}^\alpha\Gamma:\;\cA=\bigoplus_{h\in H}\cA'_h$ by setting $\cA'_h=\bigoplus_{g\in\alpha^{-1}(h)}\cA_g$. Clearly, an automorphism $\alpha$ of $U(\Gamma)$ belongs to $\W(\Gamma)$ if and only if the $U(\Gamma)$-gradings ${}^\alpha\Gamma$ and  $\Gamma$ are isomorphic.

Given gradings $\Gamma:\;\cA=\bigoplus_{g\in G}\cA_g$ and $\Gamma':\;\cA=\bigoplus_{h\in H}\cA'_h$, we say that $\Gamma'$ is a {\em coarsening} of $\Gamma$, or that $\Gamma$ is a {\em refinement} of $\Gamma'$, if for any $g\in G$ there exists $h\in H$ such that $\cA_g\subset\cA'_h$. The coarsening (or refinement) is said to be {\em proper} if the inclusion is proper for some $g$. (In particular, ${}^\alpha\Gamma$ is a coarsening of $\Gamma$, which is not necessarily proper.) A grading $\Gamma$ is said to be {\em fine} if it does not admit a proper refinement in the class of (abelian) group gradings. Any $G$-grading on a finite-dimensional algebra $\cA$ is induced from some fine grading $\Gamma$ by a homomorphism $\alpha\colon U(\Gamma)\to G$. The classification of fine gradings on $\cA$ up to equivalence is the same as the classification of maximal diagonalizable subgroupschemes of $\AAut(\cA)$ up to conjugation by $\Aut(\cA)$ (see e.g. \cite[\S 2.2]{EK_Albert}). Fine gradings on simple Lie algebras belonging to the series $A$, $B$, $C$ and $D$ (including $D_4$) were classified in \cite{E10} assuming $\FF$ algebraically closed of characteristic $0$. If we replace automorphism groups by automorphism group schemes, as was done in \cite{BK10}, then the arguments of \cite{E10} for all cases except $D_4$ (which required a completely different method) work under the much weaker assumption --- which we adopt from now on --- that $\FF$ is {\em algebraically closed of characteristic different from $2$}.


\section{Fine $\vphi$-gradings on matrix algebras}\label{se:matrix}

The goal of this section is to determine certain automorphism groups of fine $\vphi$-gradings on matrix algebras. These groups will be used in the next two sections to compute the Weyl groups of fine gradings on simple Lie algebras of series $A$, $B$, $C$ and $D$.

\subsection{Classification of fine $\vphi$-gradings on matrix algebras}

Here we present the results of \cite[\S 3]{E10} in a more explicit form. We also introduce certain objects that will appear throughout the paper.

\begin{df}
Let $\cA$ be an algebra and let $\vphi$ be an anti-automorphism of $\cA$. A $G$-grading $\Gamma:\;\cA=\bigoplus_{g\in G}\cA_g$ is said to be a {\em $\vphi$-grading} if $\vphi(\cA_g)=\cA_g$ for all $g\in G$ (i.e., $\vphi$ is an anti-automorphism of the $G$-graded algebra $\cA$) and $\vphi^2\in\Diag(\Gamma)$. The universal  group of a $\vphi$-grading is defined disregarding $\vphi$.
\end{df}

We have natural concepts of isomorphism and equivalence for $\vphi$-gradings. In addition, we will need another relation, which is weaker than equivalence.

\begin{df}
If $\Gamma_1$ is a $\vphi_1$-grading on $\cA$ and $\Gamma_2$ is a $\vphi_2$-gradings on $\cB$, we will say that $(\Gamma_1,\vphi_1)$ is {\em isomorphic} (respectively, {\em equivalent}) to $(\Gamma_2,\vphi_2)$  if there exists an isomorphism (respectively, equivalence) of graded algebras $\psi\colon\cA\to\cB$ such that $\vphi_2=\psi\vphi_1\psi^{-1}$. In the special case $\cA=\cB$ and $\vphi_1=\vphi_2$, we will simply say that $\Gamma_1$ is {\em isomorphic} (respectively, {\em equivalent}) to $\Gamma_2$. We will say that $(\Gamma_1,\vphi_1)$ is {\em weakly equivalent} to $(\Gamma_2,\vphi_2)$ if there exists an equivalence of graded algebras $\psi\colon\cA\to\cB$ such that $\xi\vphi_2=\psi\vphi_1\psi^{-1}$ for some $\xi\in\Diag(\Gamma_2)$.
\end{df}

Note that if $\vphi$ is an involution, then the condition $\vphi^2\in\Diag(\Gamma)$ is satisfied for any $\Gamma$. Also, any $\vphi$-grading $\Gamma$ on $\cA$ restricts to the space of skew-symmetric elements $\sks(\cA,\vphi)$.

Suppose $\cR$ is a matrix algebra equipped with a $G$-grading $\Gamma$. Then $\cR$ is isomorphic to $\End_\cD(V)$ where $\cD$ is a matrix algebra with a division grading (i.e., a grading that makes it a graded division algebra) and $V$ is a graded right $\cD$-module (which is necessarily free of finite rank). Let $T\subset G$ be the support of $\cD$. Then $T$ is a group and $\cD$ can be identified with a twisted group algebra $\FF^\sigma T$ for some $2$-cocycle $\sigma\colon T\times T\to\FF^\times$, i.e., $\cD$ has a basis $X_t$, $t\in T$, such that $X_u X_v=\sigma(u,v)X_{uv}$ for all $u,v\in T$ (we may assume $X_e=I$, the identity element of $\cD$). Let $\beta(u,v)=\frac{\sigma(u,v)}{\sigma(v,u)}$, so
\[
X_u X_v=\beta(u,v)X_v X_u.
\]
Then $\beta\colon T\times T\to\FF^\times$ is a nondegenerate alternating bicharacter --- see e.g. \cite[\S 2]{BK10}. A division grading on a matrix algebra with a given support $T$ and bicharacter $\beta$ can be constructed as follows. Since $\beta$ is nondegenerate and alternating, $T$ admits a ``symplectic basis'', i.e., there exists a decomposition of $T$ into the direct product of cyclic subgroups:
\begin{equation*}
T=(H_1'\times H_1'')\times\cdots\times (H_r'\times H_r'')
\end{equation*}
such that $H_i'\times H_i''$ and $H_j'\times H_j''$ are $\beta$-orthogonal for $i\neq j$, and $H_i'$ and $H_i''$ are in duality by $\beta$. Denote by $\ell_i$ the order of $H_i'$ and $H_i''$. (We may assume without loss of generality that $\ell_i$ are prime powers.) If we pick generators $a_i$ and $b_i$ for $H_i'$ and $H_i''$, respectively, then $\veps_i\bydef\beta(a_i,b_i)$ is a primitive $\ell_i$-th root of unity, and all other values of $\beta$ on the elements $a_1,b_1,\ldots,a_r,b_r$ are 1. Define the following elements of the algebra $M_{\ell_1}(\FF)\ot\cdots\ot M_{\ell_r}(\FF)$:
\begin{equation*}
X_{a_i}=I\ot\cdots I\ot X_i\ot I\ot\cdots I\quad\mbox{and}\quad X_{b_i}=I\ot\cdots I\ot Y_i\ot I\ot\cdots I,
\end{equation*}
where
\begin{equation*}
X_i=\begin{bmatrix}
\veps_i^{n-1} & 0                   & 0           & \ldots      & 0       & 0\\
0             & \veps_i^{n-2}       & 0           & \ldots      & 0       & 0\\
\ldots        &                     &             &             &         &  \\[3pt]
0             & 0                   & 0           & \ldots      & \veps_i & 0\\
0             & 0                   & 0           & \ldots      & 0       & 1
\end{bmatrix}\mbox{ and }
Y_i=\begin{bmatrix}
0 & 1 & 0 & \ldots & 0 & 0\\
0 & 0 & 1 & \ldots & 0 & 0\\
\ldots & & & & \\[3pt]
0 & 0 & 0 & \ldots & 0 & 1\\
1 & 0 & 0 & \ldots & 0 & 0
\end{bmatrix}
\end{equation*}
are the generalized Pauli matrices in the $i$-th factor, $M_{\ell_i}(\FF)$. Finally, set
\[
X_{(a_1^{i_1},b_1^{j_1},\ldots,a_r^{i_r},b_r^{j_r})}=X_{a_1}^{i_1}X_{b_1}^{j_1}\cdots X_{a_r}^{i_r}X_{b_r}^{j_r}.
\]
Identify $M_{\ell_1}(\FF)\ot\cdots\ot M_{\ell_r}(\FF)$ with $M_\ell(\FF)$, $\ell=\ell_1\cdots\ell_r=\sqrt{|T|}$, via Kronecker product. Then
\begin{equation*}
M_\ell(\FF)=\bigoplus_{t\in T}\FF X_t
\end{equation*}
is a division grading with support $T$ and bicharacter $\beta$.

Let $\vphi$ be an anti-automorphism of $\cR$ such that $\Gamma$ is a $\vphi$-grading. It is shown in \cite[\S 3]{E10} that there exists an involution $\vphi_0$ of the graded algebra $\cD$ and a $\vphi_0$-sesquilinear form $B\colon V\times V\to\cD$, which is nondegenerate, homogeneous and balanced, such that, for all $r\in\cR$, $\vphi(r)$ is the adjoint of $r$ with respect to $B$, i.e., $B(x,\vphi(r)y)=B(rx,y)$ for all $x,y\in V$ and $r\in\cR$. By {\em $\vphi_0$-sesquilinear} we mean that $B$ is $\FF$-bilinear and, for all $x,y\in V$ and $d\in\cD$, we have $B(xd,y)=\vphi_0(d)B(x,y)$ and $B(x,yd)=B(x,y)d$; by {\em balanced} we mean that, for all homogeneous $x,y\in V$, $B(x,y)=0$ is equivalent to $B(y,x)=0$. Moreover, the existence of $\vphi_0$ forces $T$ to be an elementary $2$-group. The pair $(\vphi_0,B)$ is uniquely determined by $\vphi$ up to the following transformations: for any nonzero homogeneous $d\in\cD$, we may simultaneously replace $\vphi_0$ by $\vphi'_0\colon a\mapsto d\vphi_0(a)d^{-1}$ and $B$ by $B'=dB$. Using Pauli matrices (of order $2$) as above to construct a realization of $\cD$, we see that matrix transpose $X\mapsto{}^t X$ preserves the grading: for any $u\in T$, the transpose of $X_u$ equals $\pm X_u$. It follows from \cite[Proposition 2.3]{BK10} that $(\vphi_0,B)$ can be adjusted so that $\vphi_0$ coincides with the matrix transpose. We will always assume that $(\vphi_0,B)$ is adjusted in this way, which makes $B$ unique up to a scalar in $\FF$. Also, we may write
\[
\vphi_0(X_u)=\beta(u)X_u
\]
where $\beta(u)\in\{\pm 1\}$ for all $u\in T$. If we regard $T$ as a vector space over the field of two elements, then the function $\beta\colon T\to\{\pm 1\}$ is a quadratic form whose polar form is the bicharacter $\beta\colon T\times T\to\{\pm 1\}$.

We will say that a $\vphi$-grading is {\em fine} if it is not a proper coarsening of another $\vphi$-grading. The following construction of fine $\vphi$-gradings on matrix algebras was given in \cite{E10} starting from $\cD$. We start from $T$, an elementary $2$-group of even dimension, i.e., $T=\ZZ_2^{\dim T}$, which we continue to write multiplicatively. Let $\beta$ be a nondegenerate alternating bicharacter on $T$. Fix a realization, $\cD$, of the matrix algebra endowed with a division grading with support $T$ and bicharacter $\beta$, and let $\vphi_0$ be the matrix transpose on $\cD$. Let $q\ge 0$ and $s\ge 0$ be two integers. Let
\begin{equation}\label{datum_2_tau_fine}
\tau=(t_1,\ldots,t_q),\quad t_i\in T.
\end{equation}
Denote by $\wt{G}=\wt{G}(T,q,s,\tau)$ the abelian group generated by $T$ and the symbols $\wt{g}_1,\ldots,\wt{g}_{q+2s}$ with defining relations
\begin{equation}\label{eq:def_relations_G_tilde}
\wt{g}_1^2 t_1=\ldots=\wt{g}_q^2 t_q=\wt{g}_{q+1}\wt{g}_{q+2}=\ldots=\wt{g}_{q+2s-1}\wt{g}_{q+2s}.
\end{equation}

\begin{df}\label{construct_phi_fine}
Let $\cM(\cD,q,s,\tau)$ be the $\wt{G}$-graded algebra $\End_\cD(V)$ where $V$ has a $\cD$-basis $\{v_1,\ldots,v_{q+2s}\}$ with $\deg v_i=\wt{g}_i$. Let $n=(q+2s)2^{\frac12\dim T}$ and $\cR=M_n(\FF)$. The grading $\Gamma$ on $\cR$ obtained by identifying $\cR$ with $\cM(\cD,q,s,\tau)$ will be denoted by $\M(\cD,q,s,\tau)$. In other words, we define this grading by identifying $\cR=M_{q+2s}(\cD)$ and setting $\deg(E_{ij}\ot X_t)\bydef\wt{g}_i t\wt{g}_j^{-1}$. By abuse of notation, we will also write $\M(T,q,s,\tau)$.
\end{df}

Let $\wt{G}^0$ be the subgroup of $\wt{G}$ generated by $\supp\Gamma$, which consists of the elements $z_{i,j,t}\bydef \wt{g}_i t\wt{g}_j^{-1}$, $t\in T$ (so $z_{i,i,t}=t$ for all $t\in T$). Set $z_i\bydef z_{i,i+1,e}$ for $i=1,\ldots,q$ ($i\ne q$ if $s=0$),  $z_{q+i}=z_{q+2i-1,q+2i+1,e}$ for $i=1,\ldots,s-1$, and $z_{q+s}=z_{q+1,q+2,e}$ (if $s>0$). If $s=0$, then $\wt{G}^0$ is generated by $T$ and the elements $z_1,\ldots,z_{q-1}$. If $s=1$, then $\wt{G}^0$ is generated by $T$ and $z_1,\ldots,z_{q+1}$. If $s>1$, then relations \eqref{eq:def_relations_G_tilde} imply that $z_{q+2i,q+2i+2,e}=z_{q+i}^{-1}$ for $i=1,\ldots,s-1$, hence $\wt{G}^0$ is generated by $T$ and  $z_1,\ldots,z_{q+s}$. Moreover, relations \eqref{eq:def_relations_G_tilde} are equivalent to the following:
\[
\begin{array}{lll}
z_i^2=t_i t_{i+1}\;(1\le i<q), & z_q^2 z_{q+s}=t_q\;(\mbox{if }q>0\mbox{ and }s>0).
\end{array}
\]
Let $\wt{G}^1$ be the subgroup generated by $T$ and $z_1,\ldots,z_{q-1}$. Let $\wt{G}^2$ be the subgroup generated by $z_1,\ldots,z_s$ if $q=0$ and by $z_q,\ldots,z_{q+s-1}$ if $q>0$. Then it is clear from the above relations that $\wt{G}^0=\wt{G}^1\times\wt{G}^2$, $\wt{G}^2\cong\ZZ^s$, while $\wt{G}^1=T$ if $q=0$ and $\wt{G}^1\cong\ZZ_2^{\dim T+q-1-2\dim T_0}\times\ZZ_4^{\dim T_0}$ if $q>0$, where $T_0$ is the subgroup of $T$ generated by the elements $t_i t_{i+1}$, $i=1,\ldots,q-1$. To summarize:
\begin{equation}\label{eq:univ_phi_fine}
\wt{G}^0\cong\ZZ_2^{\dim T-2\dim T_0+\max(0,q-1)}\times\ZZ_4^{\dim T_0}\times\ZZ^s.
\end{equation}
Note that relations \eqref{eq:def_relations_G_tilde} are also equivalent to the following:
\[
\begin{array}{ll}
z_{i,j,t_it}=z_{j,i,t_jt}, & i,j\le q,\quad t\in T;\\
z_{i,q+2j-1,t_it}=z_{q+2j,i,t},\quad z_{i,q+2j,t_it}=z_{q+2j-1,i,t}, & i\le q,\quad j\le s,\quad t\in T;\\
z_{q+2i-1,q+2j-1,t}=z_{q+2j,q+2i,t}, & i,j\le s,\quad t\in T;\\
z_{q+2i-1,q+2j,t}=z_{q+2j-1,q+2i,t}, & i,j\le s,\quad i\ne j,\quad t\in T.
\end{array}
\]
One verifies that, apart from the above equalities and $z_{i,i,t}=t$, the elements $z_{i,j,t}$ are distinct, so the support of $\Gamma=\M(\wt{G},\cD,\kappa,\wt{\gamma})$ is given by
\[
\begin{split}
\supp\Gamma=\, & \{z_{i,j,t}\;|\;i<j\le q,\; t\in T\} \cup \{z_{i,q+j,t}\;|\;i\le q,\; j\le 2s,\; t\in T\}\\
& \cup \{z_{q+2i-1,q+2j-1,t}\;|\;i<j\le s,\; t\in T\} \cup \{z_{q+2i,q+2j,t}\;|\;i<j\le s,\; t\in T\}\\
& \cup \{z_{q+2i-1,q+2j,t}\;|\;i,j\le s,\; i\ne j,\; t\in T\}\\
& \cup \{z_{q+2i-1,q+2i,t}\;|\;i\le s,\; t\in T\} \cup \{z_{q+2i,q+2i-1,t}\;|\;i\le s,\; t\in T\} \cup T,
\end{split}
\]
where the union is disjoint and all homogeneous components except those that appear in the last line have dimension $2$, the components of degrees $z_{q+2i-1,q+2i,t}$ and $z_{q+2i,q+2i-1,t}$ have dimension $1$, and the components of degree $t$ have dimension $q+2s$.

\begin{proposition}\label{prop:diag_grading}
Let $\Gamma=\M(\cD,q,s,\tau)$. Then $\wt{G}^0=\wt{G}^0(T,q,s,\tau)$ is the universal group of $\Gamma$, and $\Diag(\Gamma)$ consists of all automorphisms of the form $X\mapsto DXD^{-1}$, $X\in\cR$, where
\begin{equation}\label{eq:diag_grading}
D=\diag(\lambda_1,\ldots,\lambda_{q+2s})\ot X_t,\quad \lambda_i\in\FF^\times,\, t\in T,
\end{equation}
satisfying the relation
\begin{equation}\label{eq:relations_diag}
\lambda_1^2\beta(t,t_1)=\ldots=\lambda_q^2\beta(t,t_q)=\lambda_{q+1}\lambda_{q+2}=\ldots=\lambda_{q+2s-1}\lambda_{q+2s}.
\end{equation}
\end{proposition}

\begin{proof}
The relations  $z_{i,\ell,u}z_{\ell,j,v}=z_{i,j,uv}$, $u,v\in T$, can be rewritten in terms of the elements of $\supp\Gamma$, producing a set of defining relations for $\wt{G}^0$. It follows that $\wt{G}^0$ is the universal group of $\Gamma$.

Since $\wt{G}^0$ is the universal group of $\Gamma$, $\Diag(\Gamma)$ consists of all automorphisms of the form $X\mapsto \chi*X$ where $\chi$ is a character of $\wt{G}^0$. Since $\FF^\times$ is a divisible group, we can assume that $\chi$ is a character of $\wt{G}$. Let $\lambda_i=\chi(\wt{g}_i)$, $i=1,\ldots,q+2s$. Let $t$ be the element of $T$ such that $\chi(u)=\beta(t,u)$ for all $u\in T$. Looking at relations \eqref{eq:def_relations_G_tilde}, we see that \eqref{eq:relations_diag} must hold. Conversely, any $t\in T$ and a set of $\lambda_i\in\FF^\times$ satisfying \eqref{eq:relations_diag} will determine a character $\chi$ of $\wt{G}$. It remains to observe that the action of $\chi$ on $\cR$ coincides with the conjugation by $D$ as in \eqref{eq:diag_grading}.
\end{proof}

The following is Proposition 3.3 from \cite{E10}.

\begin{theorem}\label{th:E3.11}
Consider the grading $\Gamma=\M(\cD,q,s,\tau)$ on $\cR=M_{q+2s}(\cD)$ by $\wt{G}^0=\wt{G}^0(T,q,s,\tau)$ where $\tau$ is given by \eqref{datum_2_tau_fine}. Let $\mu=(\mu_1,\ldots,\mu_s)$ where $\mu_i$ are scalars in $\FF^\times$. Let $\vphi=\vphi_{\tau,\mu}$ be the anti-automorphism of $\cR$ defined by $\vphi(X)=\Phi^{-1}({}^tX)\Phi$, $X\in\cR$, where $\Phi$ is the block-diagonal matrix given by
\begin{equation}\label{formula_antaut_fine}
\Phi=\diag\left({X_{t_1},\ldots,X_{t_q}},\begin{bmatrix}0&I\\\mu_1 I&0\end{bmatrix},\ldots,\begin{bmatrix}0&I\\\mu_s I&0\end{bmatrix}\right)
\end{equation}
and $I$ is the identity element of $\cD$. Then $\Gamma$ is a fine $\vphi$-grading unless $q=2$, $s=0$ and $t_1=t_2$. In the latter case, $\Gamma$ can be refined to a $\vphi$-grading that makes $\cR$ a graded division algebra.\qed
\end{theorem}

This result and the discussion preceding Proposition 3.8 in \cite{E10} yield

\begin{theorem}\label{th:fine_phi_grad}
Let $\Gamma$ be a fine $\vphi$-grading on the matrix algebra $\cR=M_n(\FF)$ over an algebraically closed field $\FF$, $\chr{\FF}\ne 2$. Then $(\Gamma,\vphi)$ is equivalent to some $(\M(T,q,s,\tau),\vphi_{\tau,\mu})$ as in Theorem \ref{th:E3.11} where $(q+2s)2^{\frac12\dim T}=n$.\qed
\end{theorem}

In \cite{E10}, in order to obtain the classification of fine gradings on simple Lie algebras of series $A$, one classifies, {\em up to weak equivalence}, all pairs $(\Gamma,\vphi)$ where $\Gamma$ is a fine $\vphi$-grading on a matrix algebra. At the same time, for series $B$, $C$ and $D$, one classifies, {\em up to equivalence}, such pairs where $\vphi$ is an {\em involution} of appropriate type: orthogonal for series $B$ and $D$ (we write $\sgn(\vphi)=1$) and symplectic for series $C$ (we write $\sgn(\vphi)=-1$). The classifications involve equivalences $\cD\to\cD'$ satisfying certain conditions, where $\cD$ and $\cD'$ are matrix algebras with division gradings. If $T$ is the support of $\cD$ and $T'$ is the support of $\cD'$, then the graded algebras $\cD$ and $\cD'$ are equivalent if and only if the groups $T$ and $T'$ are isomorphic. Identifying $T$ and $T'$, we may assume that $\cD=\cD'$ and look at self-equivalences of $\cD$, i.e., the elements of $\Aut(\Gamma_0)$ where $\Gamma_0$ is the grading on $\cD$. By \cite[Proposition 2.7]{EK_Weyl}, the Weyl group $\W(\Gamma_0)$ is isomorphic to $\Aut(T,\beta)$, the group of automorphisms of $T$ that preserve the bicharacter $\beta$. Explicitly, if $\psi_0\in\Aut(\Gamma_0)$, then $\psi_0(X_t)\in\FF X_{\alpha(t)}$, for all $t\in T$, where $\alpha\in\Aut(T,\beta)$, and the mapping $\psi_0\mapsto\alpha$ yields an isomorphism $\Aut(\Gamma_0)/\Stab(\Gamma_0)\to\Aut(T,\beta)$. Hence the conditions in \cite{E10} can be rewritten in terms of the group $T$ rather than the graded division algebra $\cD$. Note that $\Aut(T,\beta)$ can be regarded as a sort of symplectic group; in particular, if $T$ is an elementary $2$-group, then $\Aut(T,\beta)\cong\SP_m(2)$ where $m=\dim T$.

\begin{df}\label{df:equiv_E}
Given $\tau$ as in \eqref{datum_2_tau_fine}, we will denote by $\Sigma(\tau)$ the multiset in $T$ determined by $\tau$, i.e., the underlying set of $\Sigma(\tau)$ consists of the elements that occur in $(t_1,\ldots,t_q)$, and the multiplicity of each element is the number of times it occurs there.
\end{df}

The group $\Aut(T,\beta)$ acts naturally on $T$, so we can form the semidirect product $T\rtimes \Aut(T,\beta)$, which also acts on $T$: a pair $(u,\alpha)$ sends $t\in T$ to $\alpha(t)u$. Clearly, if $\dim T=2r$, then $T\rtimes\Aut(T,\beta)$ is isomorphic to $\mathrm{ASp}_{2r}(2)$, the affine symplectic group of order $2r$ over the field of two elements (``rigid motions'' of the symplectic space of dimension $2r$).


Using this notation, Theorem 3.17 of \cite{E10} can be recast as follows:

\begin{theorem}\label{th:weak_equiv_fine_phi_grad}
Consider two pairs, $(\Gamma,\vphi)$ and $(\Gamma',\vphi')$, as in Theorem \ref{th:E3.11}, namely, $\Gamma=\M(T,q,s,\tau)$, $\vphi=\vphi_{\tau,\mu}$ and $\Gamma'=\M(T',q',s',\tau')$, $\vphi'=\vphi_{\tau',\mu'}$, where $T=\ZZ_2^{2r}$ and $T'=\ZZ_2^{2r'}$. Then $(\Gamma,\vphi)$ and $(\Gamma',\vphi')$ are weakly equivalent if and only if $r=r'$, $q=q'$, $s=s'$, and $\Sigma(\tau)$ is conjugate to $\Sigma(\tau')$ by the natural action of $T\rtimes\Aut(T,\beta)\cong\mathrm{ASp}_{2r}(2)$.\qed
\end{theorem}

Let $\psi_0\colon\cD\to\cD$ be an equivalence. Then the map $\psi_0^{-1}\vphi_0\psi_0$ is an involution of the graded algebra $\cD$, which has the same type as $\vphi_0$ (orthogonal). Hence there exists a nonzero homogeneous element $d_0\in\cD$ such that
\begin{equation}\label{eq:d_0}
d_0\vphi_0(d)d_0^{-1}=(\psi_0^{-1}\vphi_0\psi_0)(d)\quad\mbox{for all}\quad d\in\cD.
\end{equation}
Note that $d_0$ is determined up to a scalar in $\FF$. Moreover, $d_0$ is symmetric with respect to $\vphi_0$. By a similar argument, $\psi_0(d_0)$ is also symmetric. Let $\alpha$ be the element of $\Aut(T,\beta)$ corresponding to $\psi_0$ and let $t_0$ be the degree of $d_0$. Then \eqref{eq:d_0} is equivalent to the following:
\begin{equation}\label{eq:t_0}
\beta(t_0,t)\beta(t)=\beta(\alpha(t))\quad\mbox{for all}\quad t\in T,
\end{equation}
so $t_0$ depends only on $\alpha$. Moreover, $\beta(t_0)=\beta(\alpha(t_0))=1$.

\begin{df}\label{df:twisted_action}
For any $\alpha\in\Aut(T,\beta)$, the map $t\mapsto\beta(\alpha^{-1}(t))\beta(t)$ is a character of $T$, so there exists a unique element $t_\alpha\in T$ such that $\beta(t_\alpha,t)=\beta(\alpha^{-1}(t))\beta(t)$ for all $t\in T$. We define a new action of the group $\Aut(T,\beta)$ on $T$ by setting
\begin{equation*}
\alpha\cdot t\bydef\alpha(t)t_\alpha\quad\mbox{for all}\quad\alpha\in\Aut(T,\beta)\;\mbox{ and }\;t\in T.
\end{equation*}
In other words, $\Aut(T,\beta)$ acts through the (injective) homomorphism to $T\rtimes\Aut(T,\beta)$, $\alpha\mapsto(t_\alpha,\alpha)$, and the natural action of $T\rtimes\Aut(T,\beta)$ on $T$.
\end{df}

Comparing this definition with equation \eqref{eq:t_0}, which defines the element $t_0$ associated to $\alpha$, we see that $t_\alpha=\alpha(t_0)$. In particular, $\beta(t_\alpha)=1$. This implies that $\beta(\alpha\cdot t)=\beta(t)$ for all $t\in T$, so the sets
\[
T_+\bydef\{t\in T\;|\;\beta(t)=1\}\quad\mbox{and}\quad T_-\bydef\{t\in T\;|\;\beta(t)=-1\},
\]
which correspond, respectively, to symmetric and skew-symmetric homogeneous components of $\cD$ (relative to $\vphi_0$),
are invariant under the twisted action of $\Aut(T,\beta)$.


Now Proposition 3.8(2) and Theorem 3.22 of \cite{E10} can be recast as follows:

\begin{theorem}\label{th:fine_phi_grad_involution}
Let $\vphi=\vphi_{\tau,\mu}$ be as in Theorem \ref{th:E3.11}. Then $\vphi$ is an involution with $\sgn(\vphi)=\delta$ if and only if
\[
\delta=\beta(t_1)=\ldots=\beta(t_q)=\mu_1=\ldots=\mu_s.
\]
For gradings $\Gamma=\M(T,q,s,\tau)$ with $T=\ZZ_2^{2r}$ and $\Gamma'=\M(T',q',s',\tau')$ with $T'=\ZZ_2^{2r'}$ and for  involutions $\vphi=\vphi_{\tau,\mu}$ and $\vphi'=\vphi_{\tau',\mu'}$, the pairs $(\Gamma,\vphi)$ and $(\Gamma',\vphi')$ are equivalent if and only if $r=r'$, $q=q'$, $s=s'$, $\sgn(\vphi)=\sgn(\vphi')$, and $\Sigma(\tau)$ is conjugate to $\Sigma(\tau')$ by the twisted action of $\Aut(T,\beta)\cong\SP_{2r}(2)$ as in Definition \ref{df:twisted_action}.\qed
\end{theorem}

\subsection{Automorphism groups of fine $\vphi$-gradings on matrix algebras}

We are now going to study automorphisms of the fine $\vphi$-gradings $\M(T,q,s,\tau)$. We begin with some general observations. Let $\cD$ and $\cD'$ be graded division algebras, with the same grading group $G$. Let $V$ be a graded right $\cD$-module and $V'$ a graded right $\cD'$-module, both of nonzero finite rank. By an {\em isomorphism from $(\cD,V)$ to $(\cD',V')$} we mean a pair $(\psi_0,\psi_1)$ where $\psi_0\colon\cD\to\cD'$ is an isomorphism of graded algebras, $\psi_1\colon V\to V'$ is an isomorphism of graded vector spaces over $\FF$, and $\psi_1(vd)=\psi_1(v)\psi_0(d)$ for all $v\in V$ and $d\in\cD$.

Let $\cR=\End_{\cD}(V)$ and $\cR'=\End_{\cD'}(V')$. If $\psi\colon\cR\to\cR'$ is an isomorphism of graded algebras, then there exist $g\in G$ and an isomorphism $(\psi_0,\psi_1)$ from $(\cD,V^{[g]})$ to $(\cD',V')$ such that $\psi_1(rv)=\psi(r)\psi_1(v)$ for all $r\in\cR$ and $v\in V$ (see e.g. \cite[Proposition 2.5]{E10}). Here $V^{[g]}$ denotes a shift of grading: the ($\cR$,$\cD$)-bimodule structure of $V^{[g]}$ is the same as that of $V$, but we set $V^{[g]}_h=V_{hg^{-1}}$ for all $h\in G$. Conversely, given an isomorphism $(\psi_0,\psi_1)$ of the above pairs, there exists a unique isomorphism $\psi\colon\cR\to\cR'$ of graded algebras such that $\psi_1(rv)=\psi(r)\psi_1(v)$ for all $r\in\cR$ and $v\in V$. Two isomorphisms $(\psi_0,\psi_1)$ and $(\psi'_0,\psi'_1)$ determine the same isomorphism $\cR\to\cR'$ if and only if there exists a nonzero homogeneous $d\in\cD'$ such that $\psi'_0(x)=d^{-1}\psi_0(x)d$ and $\psi'_1(v)=\psi_1(v)d$ for all $x\in\cD$ and $v\in V$.

\begin{lemma}\label{equivalence_phi}
Let $\psi\colon\cR\to\cR'$ be the isomorphism of graded algebras determined by an isomorphism $(\psi_0,\psi_1)$ from $(\cD,V^{[g]})$ to $(\cD',V')$. Suppose that the graded algebras $\cR$ and $\cR'$ admit anti-automorphisms $\vphi$ and $\vphi'$, respectively, determined by a $\vphi_0$-sesquilinear form $B\colon V\times V\to\cD$ and a $\vphi'_0$-sesquilinear form $B'\colon V'\times V'\to\cD'$. Then $\vphi'=\psi\vphi\psi^{-1}$ if and only if there exists a nonzero homogeneous $d_0\in\cD$ such that
\begin{equation}\label{eq:B1_B2}
B'(\psi_1(v),\psi_1(w))=\psi_0\big(d_0 B(v,w)\big)\quad\mbox{for all}\quad v,w\in V.
\end{equation}
Moreover, $d_0\vphi_0(d)d_0^{-1}=(\psi_0^{-1}\vphi'_0\psi_0)(d)$ for all $d\in\cD$.
\end{lemma}

\begin{proof}
Set $\vphi''\bydef\psi^{-1}\vphi'\psi$ and $B''(v,w)\bydef\psi_0^{-1}\big(B'(\psi_1(v),\psi_1(w))\big)$ for all $v,w\in V$. Then we compute:
\[
\begin{split}
B''(v,wd)&=\psi_0^{-1}\big(B'(\psi_1(v),\psi_1(w)\psi_0(d))\big)\\
&=\psi_0^{-1}\big(B'(\psi_1(v),\psi_1(w))\psi_0(d)\big)=B''(v,w)d;\\
B''(vd,w)&=\psi_0^{-1}\big(B'(\psi_1(v)\psi_0(d),\psi_1(w))\big)\\
&=\psi_0^{-1}\big(\vphi'_0(\psi_0(d))B'(\psi_1(v),\psi_1(w))\big)=(\psi_0^{-1}\vphi'_0\psi_0)(d)B''(v,w);\\
B''(v,\vphi''(r)w)&=\psi_0^{-1}\big(B'(\psi_1(v),\psi(\vphi''(r))\psi_1(w))\big)\\
&=\psi_0^{-1}\big(B'(\psi_1(v),\vphi'(\psi(r))\psi_1(w))\big)\\
&=\psi_0^{-1}\big(B'(\psi(r)\psi_1(v),\psi_1(w))\big)=B''(rv,w).
\end{split}
\]
We have shown that $B''$ is a $(\psi_0^{-1}\vphi'_0\psi_0)$-sesquilinear form corresponding to $\vphi''$. Hence $\vphi''=\vphi$ if and only if there exists a nonzero homogeneous element $d_0\in\cD$ such that $B''=d_0 B$, i.e., equation \eqref{eq:B1_B2} holds.
\end{proof}

Now consider $\Gamma=\M(T,q,s,\tau)$ and $\vphi=\vphi_{\tau,\mu}$ as in Theorem \ref{th:E3.11}. There are two kinds of automorphism groups that we will need. Namely, there is
\[
\Aut^*(\Gamma,\vphi)\bydef\{\psi\in\Aut(\Gamma)\;|\;\psi\vphi\psi^{-1}=\xi\vphi\;\mbox{ for some }\;\xi\in\Diag(\Gamma)\},
\]
which will be relevant to computing the Weyl group of the corresponding fine grading on the simple Lie algebra of type $A$, and there is
\[
\Aut(\Gamma,\vphi)\bydef\{\psi\in\Aut(\Gamma)\;|\;\psi\vphi\psi^{-1}=\vphi\},
\]
which will be relevant to computing the Weyl groups of fine gradings on the simple Lie algebras of types $B$, $C$ and $D$. Hence, we are intersted in $\Aut(\Gamma,\vphi)$ only if $\vphi$ is an involution. Similarly, define
\[
\Stab(\Gamma,\vphi)\bydef\{\psi\in\Stab(\Gamma)\;|\;\psi\vphi\psi^{-1}=\vphi\}.
\]
(We could also define $\Stab^*(\Gamma,\vphi)$, but we will not need it.)

Recall that $\Gamma$ is the grading on $\cR=\End_\cD(V)$ where $\cD$ is a matrix algebra equipped with a division grading with support $T=\ZZ_2^{2r}$ and bicharacter $\beta$, and $V$ has a $\cD$-basis $\{v_1,\ldots,v_k\}$ with $\deg v_i=\wt{g}_i$ and $k=q+2s$. We will use the universal group $\wt{G}^0$ for the grading $\Gamma$. If $\psi\colon\cR\to\cR$ is an equivalence, then there exists an automorphism $\alpha$ of the group $\wt{G}^0$ such that $\psi$ sends ${}^\alpha\Gamma$ to $\Gamma$. In other words, $\psi\colon\cR'\to\cR$ is an isomorphism of graded algebras where $\cR'$ is $\cR$ as an algebra, but equipped with the grading ${}^\alpha\Gamma$. Define $\cD'$  similarly to $\cR'$, using the restriction of $\alpha$ to $T\subset\wt{G}^0$. The support of $\cD'$ is $T'=\alpha(T)$. Since  $V^{[\wt{g}_1^{-1}]}$ is $\wt{G}^0$-graded, we can also define $V'$ so that $\cR'=\End_{\cD'}(V')$ as a graded algebra. Therefore, $\psi$ is determined by $(\psi_0,\psi_1)$ where $\psi_0\colon\cD'\to\cD$ is an isomorphism of graded algebras and $\psi_1\colon V'\to V$ is an isomorphism up to a shift of grading. Hence $T'=T$ and $\psi_0\in\Aut(\Gamma_0)$, so $\psi_0(X_t)\in\FF X_{\alpha(t)}$, for all $t\in T$, and the map $\alpha\colon T\to T$ belongs to $\Aut(T,\beta)\cong \SP_{2r}(2)$. Also, if $\Psi$ is the matrix of $\psi_1$ relative to $\{v_1,\ldots,v_k\}$, we have
\[
\psi(X)=\Psi\psi_0(X)\Psi^{-1}\quad\mbox{for all}\quad X\in\cR.
\]
Since all $\wt{g}_i$ are distinct modulo $T$, matrix $\Psi$ necessarily has the form $\Psi=PD$ where $P$ is a permutation matrix and $D=\diag(d_1,\ldots,d_k)$ where $d_i$ are nonzero homogeneous elements of $\cD$. Moreover, the permutation $\pi\in\sg(k)$ corresponding to $P$ and the coset of $\psi_0$ modulo $\Stab(\Gamma_0)$ are uniquely determined by $\psi$. Hence, we have a well-defined homomorphism
\[
\Aut(\Gamma)\to\sg(k)\times\Aut(T,\beta)
\]
that sends $\psi$ to the corresponding $(\pi,\alpha)$.

Now we turn to the anti-automorphism $\vphi\colon\cR\to\cR$, which is given by the adjoint with respect to a $\vphi_0$-sesquilinear form $B$ on $V$ where $\vphi_0\colon\cD\to\cD$ is given by matrix transpose, $X_t\mapsto\beta(t)X_t$ for all $t\in T$. Recall that such $B$ is determined up to a scalar in $\FF$. We can take for $B$ the $\vphi_0$-sesquilinear form whose matrix with respect to $\{v_1,\ldots,v_k\}$ is $\Phi$ displayed in Theorem \ref{th:E3.11}. Pick $\xi\in\Diag(\Gamma)$ and let $B'$ be a $\vphi_0$-sesquilinear form on $V$ corresponding to $\xi\vphi$. By Lemma \ref{equivalence_phi}, $\psi$ satisfies $\psi\vphi\psi^{-1}=\xi\vphi$ if and only if condition \eqref{eq:B1_B2} holds for some nonzero homogeneous $d_0\in\cD$. Clearly, \eqref{eq:B1_B2} is equivalent to \eqref{eq:d_0} and
\begin{equation}\label{eq:P1_P2}
\wh{\Phi}=\psi_0(d_0\Phi),
\end{equation}
where $\wh{\Phi}$ is the matrix of $B'$ relative to $\{\psi_1(v_1),\ldots,\psi_1(v_k)\}$.
Recall that \eqref{eq:d_0} is equivalent to condition \eqref{eq:t_0} on $t_0\bydef\deg d_0$.
To summarize, $\psi$ satisfies $\psi\vphi\psi^{-1}=\xi\vphi$ if and only if
\begin{equation}\label{eq:P1_P2_repeat}
\wh{\Phi}=d_0\psi_0(\Phi)
\end{equation}
for some $d_0\in\cD$ of degree $t_\alpha$ as in Definition \ref{df:twisted_action}
(we have replaced $\psi_0(d_0)$ in \eqref{eq:P1_P2} by $d_0$ to simplify notation).

The matrix of $B'$ relative to $\{v_1,\ldots,v_k\}$ is $\Phi(D')^{-1}$ where $\xi(X)=D'X(D')^{-1}$, for all $X\in\cR$, with $D'$ of the form given by Proposition \ref{prop:diag_grading}: $D'=\diag(\nu_1 X_u,\ldots,\nu_k X_u)$ for some $u\in T$ and $\nu_i\in\FF^\times$ satisfying
\begin{equation}\label{eq:nu_i}
\nu_1^2\beta(u,t_1)=\ldots=\nu_q^2\beta(u,t_q)=\nu_{q+1}\nu_{q+2}=\ldots=\nu_{q+2s-1}\nu_{q+2s}.
\end{equation}
It follows at once that, for $\psi\in\Aut^*(\Gamma,\vphi)$, the permutation $\pi$ must preserve the set $\{1,\ldots,q\}$ and the pairing of $q+2i-1$ with $q+2i$, for $i=1,\dots,s$. It is  convenient to introduce the group $W(s)\bydef\ZZ_2^s\rtimes\sg(s)$ (i.e., the wreath product of $\sg(s)$ and $\ZZ_2$), which will be regarded as the group of permutations on $\{q+1,\ldots,q+2s\}$ that respect the block decomposition  $\{q+1,q+2\}\cup\ldots\cup\{q+2s-1,q+2s\}$. The reason for the notation $W(s)$ is that $\ZZ_2^s\rtimes\sg(s)$ is the classical Weyl group of type $B_s$ or $C_s$ (and also the extended Weyl group of type $D_s$ if $s>4$). By the above discussion, we have a homomorphism:
\begin{equation}\label{eq:homom_f1}
\Aut^*(\Gamma,\vphi)\to\sg(q)\times W(s)\times\Aut(T,\beta).
\end{equation}

We need some more notation to state the main result of this section. Let $\Sigma$ be a multiset of cardinality $q$ and let $m_1,\ldots,m_\ell$ be the multiplicities of the elements of $\Sigma$, written in some order. Thus, $m_i$ are positive integers whose sum is $q$. We will denote by $\sg\Sigma$ the subgroup $\sg(m_1)\times\cdots\times\sg(m_\ell)$ of $\sg(q)$, which may be thought of as ``interior symmetries'' of $\Sigma$. For a multiset $\Sigma$ in $T$, let $\Aut^*\Sigma$ be the stabilizer of $\Sigma$ under the natural action of $T\rtimes\Aut(T,\beta)$ on $T$, i.e., $\Aut^*\Sigma$ is the set of ``rigid motions'' of the symplectic space $T$ that permute the elements of $\Sigma$ preserving multiplicity. These are ``exterior symmetries'' of $\Sigma$. Note that each bijection $\theta\colon T\to T$ that stabilizes $\Sigma$ determines an element of $\sg(q)$ that permutes the blocks of sizes $m_1,\ldots,m_\ell$ in the same way $\theta$ permutes the elements of $\Sigma$ (thus, only blocks of equal size may be permuted) and preserves the order within each block; we will call this permutation the {\em restriction of $\theta$ to $\Sigma$}. Hence, we obtain a restriction homomorphism $\Aut^*\Sigma\to\sg(q)$. In particular, $\Aut^*\Sigma$ acts naturally on $\sg\Sigma$ by permuting factors (of equal order). Finally, let $\Aut\Sigma$ be the stabilizer of $\Sigma$ under the twisted action of $\Aut(T,\beta)$ on $T$ as in Definition \ref{df:twisted_action}. Note that $\Aut\Sigma$ may be regarded as a subgroup of $\Aut^*\Sigma$.

\begin{theorem}\label{groups_fine_phi_grad_matrix}
Let $\Gamma=\M(T,q,s,\tau)$ and let $\vphi$ be as in Theorem \ref{th:E3.11} such that $\Gamma$ is a fine $\vphi$-grading. Let $\Sigma=\Sigma(\tau)$, so $|\Sigma|=q$.
\begin{enumerate}
\item[1)] $\Stab(\Gamma,\vphi)=\Diag(\Gamma)$.
\item[2)] $\Aut^*(\Gamma,\vphi)/\Stab(\Gamma,\vphi)$ is isomorphic to an extension of the group\\ $\big((T^{q+s-1}\times\ZZ_2^s)\rtimes(\sg\Sigma\times\sg(s)\big)\rtimes \Aut^*\Sigma$
by $\ZZ_2^{q+s-1}$, with the following actions: $T^{q+s-1}$ is identified with $T^{q+s}/T$ and $\ZZ_2^{q+s-1}$ is identified with $\ZZ_2^{q+s}/\ZZ_2$, where $T$ and $\ZZ_2$ are imbedded diagonally, then
\begin{itemize}
\item $\sg\Sigma\subset\sg(q)$ acts on $T^{q+s}/T$ and $\ZZ_2^{q+s}/\ZZ_2$ by permuting the first $q$ components and trivially on $\ZZ_2^s$;
\item $\sg(s)$ acts on $T^{q+s}/T$ and $\ZZ_2^{q+s}/\ZZ_2$ by permuting the last $s$ components and naturally on $\ZZ_2^s$;
\item $\Aut^*\Sigma$ acts on $\sg\Sigma$ and $\ZZ_2^{q+s}/\ZZ_2$ through the restriction homomorphism $\Aut^*\Sigma\to\sg(q)$, trivially on $\sg(s)$, and as follows on $(T^{q+s}/T)\times\ZZ_2^s$: an element $(u,\alpha)\in\Aut^*\Sigma\subset T\rtimes\Aut(T,\beta)$ sends
a pair ${\big((u_1,\ldots,u_q,u_{q+1},\ldots,u_{q+s})T,\ul{x}\big)\in(T^{q+s}/T)\times\ZZ_2^s}$ to\\
$
\big((\alpha(u_{\pi^{-1}(1)}),\ldots,\alpha(u_{\pi^{-1}(q)}),\alpha(u_{q+1})u^{x_1},\ldots,\alpha(u_{q+s})u^{x_s})T,\ul{x}\big),
$\\
where $\pi$ is the image of $(u,\alpha)$ under the restriction homomorphism;
\item $T^{q+s-1}\times\ZZ_2^s$ acts trivially on $\ZZ_2^{q+s-1}$.
\end{itemize}
\item[3)] If $\vphi$ is an involution, then $\Aut(\Gamma,\vphi)/\Stab(\Gamma,\vphi)$ is isomorphic to\\ $\big((T^{q+s-1}\times\ZZ_2^s)\rtimes(\sg\Sigma\times\sg(s)\big)\rtimes \Aut\Sigma$, with the following actions: $T^{q+s-1}$ is identified with $T^{q+s}/T$, where $T$ is imbedded diagonally, then
\begin{itemize}
\item $\sg\Sigma\subset\sg(q)$ acts on $T^{q+s}/T$ by permuting the first $q$ components and trivially on $\ZZ_2^s$;
\item $\sg(s)$ acts on $T^{q+s}/T$ by permuting the last $s$ components and naturally on $\ZZ_2^s$;
\item $\Aut\Sigma$ acts on $\sg\Sigma$ as a subgroup of $\Aut^*\Sigma$, i.e., through the twisted action on $T$ (Definition \ref{df:twisted_action}) and restriction to $\Sigma$, trivially on $\sg(s)$, and as follows on $(T^{q+s}/T)\times\ZZ_2^s$: an element $\alpha\in\Aut\Sigma\subset\Aut(T,\beta)$ sends
a pair ${\big((u_1,\ldots,u_q,u_{q+1},\ldots,u_{q+s})T,\ul{x}\big)\in(T^{q+s}/T)\times\ZZ_2^s}$ to\\
$
\big((\alpha(u_{\pi^{-1}(1)}),\ldots,\alpha(u_{\pi^{-1}(q)}),\alpha(u_{q+1})t_\alpha^{x_1},\ldots,\alpha(u_{q+s})t_\alpha^{x_s})T,\ul{x}\big),
$\\
where $\pi$ is the image of $(t_\alpha,\alpha)$ under the restriction to $\Sigma$.
\end{itemize}
\end{enumerate}
\end{theorem}

\begin{proof}
1) If $\psi\in\Stab(\Gamma,\vphi)$, then $\Psi=PD$ where $P$ corresponds to $\pi\in\sg(q)\times W(s)$, and $\psi_0\in\Stab(\Gamma_0)$. Adjusting $D$ if necessary, we may assume $\psi_0=\id$. We claim that $\pi$ is the trivial permutation. Since $\psi$ does not permute the homogeneous components of $\Gamma$, $\pi$ must act trivially on $\wt{G}^0/T$. So, we consider the action of $\sg(q)\times W(s)$ on $\wt{G}^0/T$ in terms of the generators $z_i$ ($i=1,\ldots,q-1$ if $s=0$ and $i=1,\ldots,q+s$ if $s>0$) that were introduced after Definition \ref{construct_phi_fine}.

$\sg(q)$ acts trivially on the subgroup $\langle z_{q+1},\ldots,z_{q+s}\rangle$ and via the action of the classical Weyl group of type $A_{q-1}$, taken modulo $2$, on the subgroup $\langle z_1,\ldots,z_{q-1}\rangle\cong\ZZ_2^{q-1}$ where $z_i$ is identified with the element $\veps_i-\veps_{i+1}$, with $\{\veps_1,\ldots,\veps_q\}$ being the standard basis of $\ZZ_2^q$, on which $\sg(q)$ acts naturally.

$W(s)$ acts trivially on the subgroup $\langle z_1,\ldots,z_{q-1}\rangle$ and via the action of the classical Weyl group of type $B_s$ or $C_s$ on the subgroup $\langle z_{q+1},\ldots,z_{q+s}\rangle\cong\ZZ^s$ where $z_{q+i}$ is identified with the element $\veps_i-\veps_{i+1}$ for $i\ne s$ and $z_{q+s}$ is identified with the element $2\veps_1$, with  $\{\veps_1,\ldots,\veps_s\}$ being the standard basis of $\ZZ^s$. The easiest way to see this is to extend $\wt{G}$ by adding a new element $\wh{g}_0$ satisfying $(\wh{g}_0)^{-2}=\wt{g}_1\wt{g}_2$ and set $\wh{g}_i=\wt{g}_i\wh{g}_0$. The elements of the subgroup $\wt{G}^0$ are not affected if we replace $\wt{g}_i$ by $\wh{g}_i$, but then we have $\wh{g}_{q+2j}=\wh{g}_{q+2j-1}^{-1}$ for $j=1,\ldots,s$, so we can map $\wh{g}_{q+2j-1}$ to $\veps_j$ and $\wh{g}_{q+2j}$ to $-\veps_j$.

Note that the action of $W(s)$ on $\langle z_{q+1},\ldots,z_{q+s}\rangle$ is always faithful, while the action of $\sg(q)$ on $\langle z_1,\ldots,z_{q-1}\rangle$ is faithful unless $q=2$. If $q>0$ and $s>0$, then we also have the generator $z_q$, on which $\pi\in\sg(q)\times W(s)$ acts in this way (note that $\pi(q)\le q$ and $\pi(q+1)>q$):
\[
z_q\mapsto\left\{\begin{array}{ll}
z_{\pi(q)}\cdots z_q z_{q+1}\cdots z_{q+j}&\mbox{if }\pi(q+1)=q+2j+1;\\
z_{\pi(q)}\cdots z_q z_{q+1}^{-1}\cdots z_{q+j}^{-1}z_{q+s}&\mbox{if }\pi(q+1)=q+2j+2.
\end{array}\right.
\]
If $\pi$ acts trivially on $\langle z_{q+1},\ldots,z_{q+s}\rangle$, then $\pi(q+1)=q+1$. Hence, if $\pi$ also acts trivially on $z_q$, then $\pi(q)=q$. It follows that the action of $\sg(q)\times W(s)$ on $\wt{G}^0/T$ is faithful unless $q=2$ and $s=0$. In this remaining case, we have $\tau=(t_1,t_2)$ where $t_1\ne t_2$ (otherwise $\Gamma$ is not a fine $\vphi$-grading). If $\psi_1$ yields $\pi=(12)$, then $\psi_1(v_1)=v_2 d_1$ and $\psi(v_2)=v_1 d_2$ for some nonzero homogeneous $d_1,d_2\in\cD$, but then $B(\psi_1(v_1),\psi_1(v_1))$ has degree $t_2$, while $B(v_1,v_1)$ has degree $t_1$. This contradicts \eqref{eq:P1_P2_repeat}, because here we have $\psi_0=\id$, $d_0\in\FF^\times$ and $B'=B$. The proof of the claim is complete.

Since $P=I$, we have $\Psi=\diag(d_1,\ldots,d_k)$, where the $d_i$ must necessarily have the same degree, say, $t$, so $\Psi=\diag(\lambda_1,\ldots,\lambda_k)\ot X_t$, but then \eqref{eq:P1_P2_repeat} implies that \eqref{eq:relations_diag} must hold, hence $\psi\in\Diag(\Gamma)$. We have proved that $\Stab(\Gamma,\vphi)\subset\Diag(\Gamma)$. The opposite inclusion is obvious.

2) We can extract more information about an element $\psi\in\Aut^*(\Gamma,\vphi)$ than given by its image under the homomorphism \eqref{eq:homom_f1} if we look at the action of $\psi$ on $\vphi$. Write $\psi\vphi\psi^{-1}=\xi_\psi\vphi$ where $\xi_\psi$ is a uniquely determined element of $\Diag(\Gamma)$. Clearly, we have $\xi_{\psi\psi'}=\xi_{\psi}(\psi\xi_{\psi'}\psi^{-1})$. Since $\xi_\psi$ is the conjugation by $\diag(\nu_1,\ldots,\nu_k)\ot X_{u_\psi}$, for a uniquely determined $u_\psi\in T$, we obtain $u_{\psi\psi'}=u_\psi\alpha_\psi(u_{\psi'})$ where $\alpha_\psi$ is the element of $\Aut(T,\beta)$ corresponding to $\psi$ under \eqref{eq:homom_f1}. Hence, we can construct a homomorphism
\begin{equation}\label{eq:homom_f2}
\Aut^*(\Gamma,\vphi)\to\sg(q)\times W(s)\times (T\rtimes\Aut(T,\beta)),
\end{equation}
where the first two components are as in \eqref{eq:homom_f1} and the third is $\psi\mapsto(u_\psi,\alpha_\psi)$.

Theorem \ref{th:weak_equiv_fine_phi_grad} implies that we may assume without loss of generality that
\[
\Phi=\diag\left({X_{t_1},\ldots,X_{t_q}},\begin{bmatrix}0&I\\ I&0\end{bmatrix},\ldots,\begin{bmatrix}0&I\\I&0\end{bmatrix}\right).
\]
(In other words, the scalars $\mu_i$ are all equal to $1$.) Then, for $\psi$ given by $\Psi=PD$ and $\psi_0\in\Aut(\Gamma_0)$, with $P$ corresponding to $\pi\in\sg(q)\times W(s)$, condition \eqref{eq:P1_P2_repeat} is equivalent to the following, with $u=u_\psi$:
\begin{equation}\label{eq:Phi1}
\vphi_0(d_i)X_{t_{\pi(i)}}\nu_{\pi(i)}^{-1}X_u^{-1}d_i=d_0\psi_0(X_{t_i}),\quad i=1,\ldots,q,
\end{equation}
and, for each $j=1,\ldots,s$, one of the following depending on whether $\pi(q+2j-1)<\pi(q+2j)$ or $\pi(q+2j-1)>\pi(q+2j)$:
\begin{equation}\label{eq:Phi2}
\vphi_0(d_{q+2j-1})\nu_{\pi(q+2j)}^{-1}X_u^{-1}d_{q+2j}=\vphi_0(d_{q+2j})\nu_{\pi(q+2j-1)}^{-1}X_u^{-1}d_{q+2j-1}=d_0
\end{equation}
in the first case, and
\begin{equation}\label{eq:Phi3}
\vphi_0(d_{q+2j-1})\nu_{\pi(q+2j-1)}^{-1}X_u^{-1}d_{q+2j}=\vphi_0(d_{q+2j})\nu_{\pi(q+2j)}^{-1}X_u^{-1}d_{q+2j-1}=d_0
\end{equation}
in the second case.

If $\psi\in\Aut^*(\Gamma,\vphi)$, then, looking at the degrees in \eqref{eq:Phi1}, we obtain
\begin{equation}\label{eq:Phi1_deg}
t_{\pi(i)}=\alpha_\psi(t_i)t_{\alpha_\psi} u_\psi,\quad i=1,\ldots,q,
\end{equation}
which implies that $(t_{\alpha_\psi} u_\psi,\alpha_\psi)$ belongs to $\Aut^*\Sigma$. Composing the third component of the homomorphism \eqref{eq:homom_f2} with the automorphism $(u,\alpha)\mapsto(t_\alpha u,\alpha)$ of the group $T\rtimes\Aut(T,\beta)$, we obtain a homomorphism
\begin{equation}\label{eq:homom_f3}
\Aut^*(\Gamma,\vphi)\to\sg(q)\times W(s)\times\Aut^*\Sigma.
\end{equation}
For any element $(t_\alpha u,\alpha)\in\Aut^*\Sigma$, let $\pi_{u,\alpha}\in\sg(q)$ be its restriction to $\Sigma$. Then  \eqref{eq:Phi1_deg} implies that the permutation $\pi\pi_{u_\psi,\alpha_\psi}^{-1}$ does not move the elements of the underlying set of $\Sigma$, so it belongs to $\sg\Sigma$. It follows that \eqref{eq:homom_f3} can be rearranged as follows:
\begin{equation*}
f\colon\Aut^*(\Gamma,\vphi)\to W(s)\times(\sg\Sigma\rtimes\Aut^*\Sigma).
\end{equation*}
We claim that $f$ is surjective. We will construct representatives in $\Aut^*(\Gamma,\vphi)$ for the elements of each of the subgroups $W(s)$, $\sg\Sigma$ and $\Aut^*\Sigma$.

For any $\pi\in W(s)$, let $P$ be the corresponding permutation matrix and let $\psi_\pi$ be given by $\Psi=P$ and $\psi_0=\id$. Let $\alpha$ be the automorphism of $\wt{G}$ that restricts to identity on $T$ and sends $\wt{g}_i$ to $\wt{g}_{\pi(i)}$ (in particular, $\wt{g}_i$ are fixed for $i=1,\ldots,q$). Then $\psi_\pi$ sends ${}^\alpha\Gamma$ to $\Gamma$, so $\psi_\pi\in\Aut(\Gamma)$. Also, conditions \eqref{eq:Phi1} through \eqref{eq:Phi3} are satisfied with $d_0=I$, $u=e$ and $\nu_i=1$, so $\psi_\pi\in\Aut(\Gamma,\vphi)$.

For any $\pi\in\sg\Sigma$, let $P$ be the corresponding permutation matrix and let $\psi_\pi$ be given by $\Psi=P$ and $\psi_0=\id$. Since we have $t_{\pi(i)}=t_i$ for all $i=1,\ldots,q$, we can define the automorphism $\alpha$ of $\wt{G}$ in the same way as above (this time, $\wt{g}_i$ are fixed for $i=q+1,\ldots,q+2s$). Then $\psi_\pi$ sends ${}^\alpha\Gamma$ to $\Gamma$, so $\psi_\pi\in\Aut(\Gamma)$. Also, conditions \eqref{eq:Phi1} and \eqref{eq:Phi2} are satisfied with $d_0=I$, $u=e$ and $\nu_i=1$, so $\psi_\pi\in\Aut(\Gamma,\vphi)$.

Now, for any $(t_\alpha u,\alpha)\in\Aut^*\Sigma$, let $\pi=\pi_{u,\alpha}$. Then $t_{\pi(i)}=\alpha(t_i)t_{\alpha} u$ for $i=1,\ldots,q$ and hence we can extend $\alpha\colon T\to T$ to an automorphism of $\wt{G}$ by setting $\alpha(\wt{g}_i)=\wt{g}_{\pi(i)}$ for $i=1,\ldots,q$, $\alpha(\wt{g}_{q+2j-1})=\wt{g}_{q+2j-1}$ and $\alpha(\wt{g}_{q+2j})=\wt{g}_{q+2j}t_\alpha u$ for $j=1,\ldots,s$. Choose $\nu_i\in\FF^\times$ such that $\nu_i^2=\beta(u,t_i)\beta(u)$, $i=1,\ldots,q$, and set $\nu_{q+2j}=1$ and $\nu_{q+2j-1}=\beta(u)$, $j=1,\ldots,s$. Then \eqref{eq:nu_i} holds, so the conjugation by $\diag(\nu_1 X_u,\ldots,\nu_k X_u)$ is an element $\xi\in\Diag(\Gamma)$. Choose $\psi_0$ such that $\psi_0(X_t)\in\FF X_{\alpha(t)}$. Let $P$ be the permutation matrix corresponding to $\pi$ and let
\[
D=\diag(\lambda_1 I,\ldots,\lambda_q I,I,X_u X_{t_\alpha},\ldots,I,X_u X_{t_\alpha}),
\]
where $\lambda_i\in\FF^\times$ are selected in such a way that condition \eqref{eq:Phi1} holds with $d_0=X_{t_\alpha}$ (the degrees of both sides match, so it is indeed possible to find such $\lambda_i$). Since $\beta(t_\alpha)=1$, condition \eqref{eq:Phi2} also holds. Finally, let $\psi_{u,\alpha}$ be given by $\Psi=PD$ and $\psi_0$. Then $\psi_{u,\alpha}$ sends ${}^\alpha\Gamma$ to $\Gamma$ and $\vphi$ to $\xi\vphi$, with $\alpha$ and $\xi$ indicated above. Therefore, $\psi_{u,\alpha}$ belongs to $\Aut^*(\Gamma,\vphi)$.

We have proved that the homomorphism $f$ is surjective. Let $K$ be the kernel of $f$. It consists of the conjugations by  matrices of the form $D=\diag(d_1,\ldots,d_k)$ such that \eqref{eq:Phi1} and \eqref{eq:Phi2} are satisfied with $\pi=\id$, $\psi_0=\id$, $d_0\in\FF^\times$ and $u=e$. Hence $\deg d_{q+2j-1}=\deg d_{q+2j}$ for all $j=1,\ldots,s$. Conversely, given $(u_1,\ldots,u_k)\in T^k$ with $u_{q+2j-1}=u_{q+2j}$ for $j=1,\ldots,s$, we can find elements $d_i$ with $\deg d_i=u_i$ such that the conjugation by $D$ belongs to $\Aut(\Gamma,\vphi)$.

According to 1), the subgroup
\[
N=\{\psi\in K\;|\;\deg d_1=\cdots=\deg d_k\}
\]
contains $\Stab(\Gamma,\vphi)$. Clearly, $N$ is normal in $\Aut^*(\Gamma,\vphi)$. From the previous paragraph it follows that $K/N\cong T^{q+s}/T$ where $T$ is imbedded into $T^{q+s}$ diagonally. The representatives $\psi_\pi$ that we constructed above for $\pi\in W(s)$ and for $\pi\in\sg\Sigma$ form subgroups of $\Aut(\Gamma,\vphi)$ that commute with one another. But observe also that the representatives $\psi_{u,\alpha}$ for $(t_\alpha u,\alpha)\in\Aut^*\Sigma$ form a subgroup modulo $N$. Moreover, for $\pi\in\sg(s)\subset W(s)$ the elements $\psi_{u,\alpha}$ and $\psi_\pi$ commute modulo $N$, while for $\pi\in\sg\Sigma$ we have $\psi_{u,\alpha}\psi_\pi\psi_{u,\alpha}^{-1}\in\psi_{\pi_{u,\alpha}\pi\pi_{u,\alpha}^{-1}} N$. Finally, for the transposition $\pi=(q+2j-1,q+2j)$, we have $\psi_\pi\psi_{u,\alpha}\psi_\pi\psi_{u,\alpha}^{-1}\in\psi N$ where $\psi$ is the conjugation by $\diag(d_1,\ldots,d_k)$ with $d_{q+2j-1}=d_{q+2j}=X_{t_\alpha u}$ and all other $d_i=I$. It follows that $\Aut^*(\Gamma,\vphi)/N$ is isomorphic to $\big((T^{q+s-1}\times\ZZ_2^s)\rtimes(\sg\Sigma\times\sg(s)\big)\rtimes \Aut^*\Sigma$, with the stated actions.

It remains to compute the quotient $N/\Stab(\Gamma,\vphi)$. Since any element $\psi\in N$ belongs to $\Stab(\Gamma)$, the mapping $\psi\mapsto\xi_\psi$ is a homomorphism $N\to\Diag(\Gamma)$ whose kernel is exactly $\Stab(\Gamma,\vphi)$. Hence, it suffices to compute the image. Since here $u=e$ and $\deg d_{q+2j-1}=\deg d_{q+2j}$, condition \eqref{eq:Phi2} implies that $\nu_{q+2j-1}=\nu_{q+2j}$ for $j=1,\ldots,s$. But then \eqref{eq:nu_i} implies that all $\nu_i^2$ are equal to each other. Since multiplying all $\nu_i$ by the same scalar in $\FF^\times$ does not change $\xi$, we may assume that $\nu_i\in\{\pm 1\}$. In fact, for $D=\diag(\lambda_1 I,\ldots,\lambda_k I)$, conditions \eqref{eq:Phi1} and \eqref{eq:Phi2} reduce to the following: up to a common scalar multiple, $\nu_i=\lambda_i^2$ for $i=1,\ldots,q$, and $\nu_{q+2j-1}=\nu_{q+2j}=\lambda_{q+2j-1}\lambda_{q+2j}$ for $j=1,\ldots,s$. Hence every $(\nu_1,\ldots,\nu_k)$ with $\nu_i\in\{\pm 1\}$ and $\nu_{q+2j-1}=\nu_{q+2j}$ indeed appears in $\xi_\psi$ for some $\psi\in N$. Therefore, the quotient $N/\Stab(\Gamma,\vphi)$ is isomorphic to $\ZZ_2^{q+s}/\ZZ_2$ where $\ZZ_2$ is imbedded into $\ZZ_2^{q+s}$ diagonally.

3) The proof is similar to 2), so we will merely point out the differences. According to Theorem \ref{th:fine_phi_grad_involution}, here we have
\[
\Phi=\diag\left({X_{t_1},\ldots,X_{t_q}},\begin{bmatrix}0&I\\ \delta I&0\end{bmatrix},\ldots,
\begin{bmatrix}0&I\\ \delta I&0\end{bmatrix}\right),
\]
where $\delta=\sgn(\vphi)$ and $\beta(t_i)=\delta$ for $i=1,\ldots,q$. Also, $B'$ equals $B$ and hence, for $\psi$ given by $\Psi=PD$ and $\psi_0\in\Aut(\Gamma_0)$, with $P$ corresponding to $\pi\in\sg(q)\times W(s)$, condition \eqref{eq:P1_P2_repeat} is equivalent to the following:
\begin{equation}\label{eq:Phi1_}
\vphi_0(d_i)X_{t_{\pi(i)}}d_i=d_0\psi_0(X_{t_i}),\quad i=1,\ldots,q,
\end{equation}
and, for each $j=1,\ldots,s$, one of the following depending on whether $\pi(q+2j-1)<\pi(q+2j)$ or $\pi(q+2j-1)>\pi(q+2j)$:
\begin{equation}\label{eq:Phi2_}
\vphi_0(d_{q+2j-1})d_{q+2j}=d_0
\end{equation}
in the first case, and
\begin{equation}\label{eq:Phi3_}
\vphi_0(d_{q+2j-1})d_{q+2j}=\delta d_0
\end{equation}
in the second case. Here we took into account that, since $\vphi_0(d_0)=d_0$, either \eqref{eq:Phi2_} or \eqref{eq:Phi3_} implies $\vphi_0(d_{q+2j-1})d_{q+2j}=\vphi_0(d_{q+2j})d_{q+2j-1}$.

If $\psi\in\Aut(\Gamma,\vphi)$, then, looking at the degrees in \eqref{eq:Phi1_}, we obtain
\begin{equation}\label{eq:Phi1_deg_}
t_{\pi(i)}=\alpha_\psi(t_i)t_{\alpha_\psi},\quad i=1,\ldots,q,
\end{equation}
which implies that $(t_{\alpha_\psi},\alpha_\psi)$ stabilizes $\Sigma$, i.e., $\alpha_\psi$ belongs to $\Aut\Sigma$. Hence we obtain a homomorphism
\begin{equation}\label{eq:homom_f3_}
\Aut(\Gamma,\vphi)\to\sg(q)\times W(s)\times\Aut\Sigma.
\end{equation}
For any element $\alpha\in\Aut\Sigma$, let $\pi_{\alpha}\in\sg(q)$ be the restriction of its twisted action to $\Sigma$. Then  \eqref{eq:Phi1_deg_} implies that the permutation $\pi\pi_{\alpha_\psi}^{-1}$ does not move the elements of the underlying set of $\Sigma$, so it belongs to $\sg\Sigma$. It follows that \eqref{eq:homom_f3_} can be rearranged as follows:
\begin{equation*}
f\colon\Aut(\Gamma,\vphi)\to W(s)\times(\sg\Sigma\rtimes\Aut\Sigma).
\end{equation*}
To prove that $f$ is surjective, we construct representatives in $\Aut(\Gamma,\vphi)$ for the elements of each of the subgroups $W(s)$, $\sg\Sigma$ and $\Aut\Sigma$.

For $\pi$ in $\sg\Sigma$ or in $\sg(s)\subset W(s)$, we take the same representatives as in the proof of 2). For $\pi=(q+2j-1,q+2j)\in W(s)$, a slight modification is needed: we take $\Psi=PD$ rather than just $P$, where $d_{q+2j}=\delta I$ and all other $d_i=I$. For any $\alpha\in\Aut\Sigma$, let $\pi=\pi_{\alpha}$. Then $t_{\pi(i)}=\alpha(t_i)t_{\alpha}$ for $i=1,\ldots,q$ and hence we can extend $\alpha\colon T\to T$ to an automorphism of $\wt{G}$ by setting $\alpha(\wt{g}_i)=\wt{g}_{\pi(i)}$ for $i=1,\ldots,q$, $\alpha(\wt{g}_{q+2j-1})=\wt{g}_{q+2j-1}$ and $\alpha(\wt{g}_{q+2j})=\wt{g}_{q+2j}t_\alpha$ for $j=1,\ldots,s$. Choose $\psi_0$ such that $\psi_0(X_t)\in\FF X_{\alpha(t)}$. Let $P$ be the permutation matrix corresponding to $\pi$ and let
\[
D=\diag(\lambda_1 I,\ldots,\lambda_q I,I,X_{t_\alpha},\ldots,I,X_{t_\alpha}),
\]
where $\lambda_i\in\FF^\times$ are selected in such a way that condition \eqref{eq:Phi1_} holds with $d_0=X_{t_\alpha}$. Clearly, condition \eqref{eq:Phi2_} also holds. Finally, let $\psi_{\alpha}$ be given by $\Psi=PD$ and $\psi_0$. Then $\psi_{\alpha}$ sends ${}^\alpha\Gamma$ to $\Gamma$ and fixes $\vphi$, so $\psi_{\alpha}$ belongs to $\Aut(\Gamma,\vphi)$.

Let $K$ be the kernel of $f$ and let
\[
N=\{\psi\in K\;|\;\deg d_1=\cdots=\deg d_k\}.
\]
The same arguments as in 2) show that $K/N\cong T^{q+s}/T$ and $\Aut(\Gamma,\vphi)/N$ is isomorphic to $\big((T^{q+s-1}\times\ZZ_2^s)\rtimes(\sg\Sigma\times\sg(s)\big)\rtimes \Aut\Sigma$, with the stated actions. But here we have $N=\Stab(\Gamma,\vphi)$, which completes the proof.
\end{proof}


\section{Series $A$}\label{se:A}

In this section we describe the Weyl groups of fine gradings on the simple Lie algebras of series $A$. Thus, we take $\cR=M_n(\FF)$, $n\ge 2$, and $\cL=\Psl_n(\FF)=[\cR,\cR]/(Z(\cR)\cap[\cR,\cR])$. First we review the classification of fine gradings on $\cL$ from \cite{E10} (extended to positive characteristic using automorphism group schemes) and then derive the Weyl groups for $\cL$ from what we already know about automorphisms of fine gradings (\cite{EK_Weyl}) and fine $\vphi$-gradings (Section \ref{se:matrix}) on $\cR$.

\subsection{Classification of fine gradings}

The case $n=2$ is easy, because the restriction from $\cR$ to $\cL$ yields an isomorphism $\AAut(\cR)\to\AAut(\cL)$. It follows that the classification of fine gradings on $\cL$ is the same as that on $\cR$. Namely, there are two fine gradings on $\Sl_2(\FF)$, up to equivalence: the Cartan grading, whose universal group is $\ZZ$, and the Pauli grading, whose universal group is $\ZZ_2^2$.

Now assume $n\ge 3$. Then the restriction and passing modulo the center yields a closed imbedding $\AAut(\cR)\to\AAut(\cL)$, which is not an isomorphism. To rectify this, one introduces the affine group scheme $\overline{\AAut}(\cR)$ corresponding to the algebraic group of automorphisms and anti-automorphisms of $\cR$ (see \cite[\S 3]{BK10}). Unless $n=\chr{\FF}=3$, we obtain an isomorphism $\overline{\AAut}(\cR)\to\AAut(\cL)$. It is convenient to divide gradings on $\cL$ into two types: for Type I the corresponding diagonalizable subgroupscheme of $\AAut(\cL)$ is contained in the image of the closed imbedding $\AAut(\cR)\to\AAut(\cL)$, while for Type II it is not. In other words, a grading on $\cL$ is of Type I if and only if it is induced from a (unique) grading on $\cR$ by restriction and passing modulo the center.

In \cite{BK10}, the {\em distinguished element} of a Type II grading $\Gamma$ is introduced. It can be characterized as the unique element $h$ of order $2$ in the grading group $G$ such that the coarsening $\wb{\Gamma}$ induced from $\Gamma$ by the quotient map $G\to\bG\bydef G/\langle h\rangle$ is a Type I grading. The original grading $\Gamma$ can be recovered from $\wb{\Gamma}$ if we know the action of some character $\chi$ of $G$ with $\chi(h)=-1$. Indeed, we just have to split each component of $\wb{\Gamma}$ into eigenspaces with respect to the action of $\chi$. We can transfer this procedure to $\cR$ in the following way. The action of $\chi$ on $\cL$ is induced by $-\vphi$ where $\vphi$ is an anti-automorphism of $\cR$. The Type I grading $\wb{\Gamma}$ on $\cL$ comes from a grading $\wb{\Gamma}'$ on $\cR$. Since $-\vphi$ is an automorphism of $\brac{\cR}$ (the Lie algebra $\cR$ under commutator) and $\vphi^2$ acts as a scalar on each component of $\wb{\Gamma}'$, we can refine the $\bG$-grading $\wb{\Gamma}':\;\cR=\bigoplus_{\bg\in\bG}\cR_\bg$ to a $G$-grading $\Gamma':\;\brac{\cR}=\bigoplus_{g\in G}\cR_g$ by splitting each component $\cR_\bg$ into eigenspaces of $\vphi$. In detail, $\vphi^2$ acts on $\cR_\bg$ as multiplication by $\chi^2(\bg)$ (where we regard $\chi^2$ as a character of $\bG$, since $\chi^2(h)=1$), so we set
\begin{equation}\label{eq:refine_bar_G_repeat}
\cR_g=\{X\in \cR_\bg\;|\;\vphi(X)=-\chi(g)X\}=\{\vphi(X)-\chi(g)X\;|\;X\in \cR_\bg\}.
\end{equation}
Then $\Gamma'$ induces the original Type II grading $\Gamma$ on $\cL$ by restriction and passing modulo the center.

Now we apply the above to fine gradings on $\cL$. The fine gradings of Type I come from the fine gradings on $\cR$ that do not admit an anti-automorphism $\vphi$ making them $\vphi$-gradings. All fine gradings on $\cR$ are obtained as follows. We start from $T$, a finite  abelian group that admits a nondegenerate alternating bicharacter $\beta$ (hence $|T|$ is a square). Fix a realization, $\cD$, of the matrix algebra endowed with a division grading with support $T$ and bicharacter $\beta$. Let $k\ge 1$ be an integer. Denote by $\wt{G}=\wt{G}(T,k)$ the abelian group freely generated by $T$ and the symbols $\wt{g}_1,\ldots,\wt{g}_k$.

\begin{df}\label{construct_fine}
Let $\cM(\cD,k)$ be the $\wt{G}$-graded algebra $\End_\cD(V)$ where $V$ has a $\cD$-basis $\{v_1,\ldots,v_k\}$ with $\deg v_i=\wt{g}_i$. Let $n=k\sqrt{|T|}$ and $\cR=M_n(\FF)$. The grading on $\cR$ obtained by identifying $\cR$ with $\cM(\cD,k)$ will be denoted by $\M(\cD,k)$. In other words, we define this grading by identifying $\cR=M_k(\cD)$ and setting $\deg(E_{ij}\ot X_t)\bydef\wt{g}_i t\wt{g}_j^{-1}$. By abuse of notation, we will also write $\M(T,k)$.
\end{df}

The universal group of $\M(T,k)$ is the subgroup $\wt{G}^0=\wt{G}(T,k)^0$ of $\wt{G}$ generated by the support, i.e., by the elements $z_{i,j,t}\bydef \wt{g}_i t\wt{g}_j^{-1}$, $t\in T$. Clearly, $\wt{G}^0\cong T\times\ZZ^{k-1}$. By \cite[Proposition 3.24]{E10}, $\M(T,k)$ is a $\vphi$-grading for some $\vphi$ if and only if $T$ is an elementary $2$-group and $k\le 2$. Two gradings, $\M(T,k)$ and $\M(T',k')$, are equivalent if and only if $T\cong T'$ and $k=k'$.

\begin{df}\label{type_I_fine}
Consider the grading $\M(T,k)$ on $\cR$ by the group $\wt{G}(T,k)^0$ where $k\ge 3$ if $T$ is an elementary $2$-group. The $\wt{G}(T,k)^0$-grading on $\cL$ obtained by restriction and passing modulo the center will be denoted by $\AI(T,k)$.
\end{df}

The grading $\AI(T,k)$ is fine, and $\wt{G}(T,k)^0$ is its universal group. To deal with fine gradings of Type II, we will need the following general observation:

\begin{lemma}\label{lm:extend_group}
Let $\wb{\Gamma}$ be a $\vphi$-grading on an algebra $\cA$ and let $\bG$ be its universal group. Then there exist an abelian group $G$, an element $h\in G$ of order $2$, a character $\chi$ of $G$ with $\chi(h)=-1$ such that $\bG=G/\langle h\rangle$ and the action of $\chi^2$ on the $\bG$-graded algebra $\cA$ (regarding $\chi^2$ as a character of the group $\bG$) coincides with $\vphi^2$. The pair $(G,h)$ is determined uniquely up to isomorphism over $\bG$ (i.e., $\langle h\rangle\to G\to\bG$ is unique up to equivalence of extensions).
\end{lemma}

\begin{proof}
For each $\bg\in\bG$, $\vphi^2$ acts on $\cA_{\bg}$ as multiplication by some $\lambda(\bg)\in\FF^\times$. Since $\bG$ is the universal group of $\wb{\Gamma}$, $\lambda\colon\bG\to\FF^\times$ is a homomorphism. For each $\bg\in\bG$, we select $\mu(\bg)\in\FF^\times$ such that $\mu(\bg)^2=\lambda(\bg)$ (there are two choices). It will be convenient to choose $\mu(\be)=1$. It follows that
\begin{equation}\label{eq:extension_factor}
\mu(\wb{x}\,\wb{y})=\veps(\wb{x},\wb{y})\mu(\wb{x})\mu(\wb{y})\quad\mbox{for all}\quad\wb{x},\wb{y}\in\bG
\end{equation}
where $\veps(\wb{x},\wb{y})\in\{\pm 1\}$. One immediately verifies that $\veps$ is a symmetric $2$-cocycle on $\bG$ with $\veps(\bg,\be)=1$ for all $\bg\in\bG$ and, moreover, the class of $\veps$ in $H^2(\bG,\ZZ_2)$ (where we identified $\{\pm 1\}$ with $\ZZ_2$) does not depend on the choices of $\mu(\bg)$. Let $G$ be the central extension of $\bG$ by $\ZZ_2$ determined by $\veps$, i.e., $G$ consists of the pairs $(\bg,\delta)$, $\bg\in\bG$, $\delta\in\{\pm 1\}$, with multiplication given by
\begin{equation}\label{eq:extension_multiplication}
(\wb{x},\delta_1)(\wb{y},\delta_2)=(\wb{x}\,\wb{y},\,\veps(\wb{x},\wb{y})\delta_1\delta_2)\quad\mbox{for all}\quad\wb{x},\wb{y}\in\bG\;\mbox{ and }\;\delta_1,\delta_2\in\{\pm 1\}.
\end{equation}
Define $\chi\colon G\to\FF^\times$ by $(\bg,\delta)\mapsto\mu(\bg)\delta$. Comparing \eqref{eq:extension_factor} and \eqref{eq:extension_multiplication}, we see that $\chi$ is a homomorphism. Set $h=(\be,-1)\in G$. Then $h$ has order $2$ and $\chi(h)=-1$.  By construction, the action of $\chi^2$ on $\cA$ determined by $\wb{\Gamma}$ coincides with $\vphi^2$.
\end{proof}

Let $T$ be an elementary $2$-group of even dimension. Recall the group $\wt{G}(T,q,s,\tau)$, which was introduced before Definition \ref{construct_phi_fine}, and its subgroup $\wt{G}(T,q,s,\tau)^0$.

\begin{df}\label{type_II_fine}
Consider the grading $\wb{\Gamma}=\M(T,q,s,\tau)$ on $\cR$ by the group $\bG=\wt{G}(T,q,s,\tau)^0$ where $t_1\ne t_2$ if $q=2$ and $s=0$. Let $\Phi$ be the matrix given by
\begin{equation*}
\Phi=
\diag\left({X_{t_1},\ldots,X_{t_q}},\begin{bmatrix}0&I\\I&0\end{bmatrix},\ldots,\begin{bmatrix}0&I\\I&0\end{bmatrix}\right).
\end{equation*}
Define $\vphi(X)=\Phi^{-1}({}^tX)\Phi$. Let $G$, $h$ and $\chi$ be as in Lemma \ref{lm:extend_group}, so we obtain a $G$-grading on $\brac{\cR}$ defined by \eqref{eq:refine_bar_G_repeat}. The $G$-grading on $\cL$ obtained by restriction and passing modulo the center will be denoted by $\AII(T,q,s,\tau)$.
\end{df}

The grading $\AII(T,q,s,\tau)$ is fine, and $G$ is its universal group. Note that $\vphi^4=\id$. It can be shown (cf. \cite[Example 3.21]{E10}) that the extension $\langle h\rangle\to G\to\bG$ is split if and only if there exists $t\in T$ such that $t_i t$ are in $T_+$ for all $i$ or in $T_-$ for all $i$. Taking into account \eqref{eq:univ_phi_fine}, we see that $G$ is isomorphic to
\[
\left\{\begin{array}{l}
\ZZ_2^{\dim T-2\dim T_0+\max(0,q-1)+1}\times\ZZ_4^{\dim T_0}\times\ZZ^s\quad\mbox{if}\quad\exists t\in T\quad\beta(t_1 t)=\ldots=\beta(t_q t);\\
\ZZ_2^{\dim T-2\dim T_0+\max(0,q-1)}\times\ZZ_4^{\dim T_0+1}\times\ZZ^s\quad\mbox{otherwise},
\end{array}\right.
\]
where $T_0$ is the subgroup of $T$ generated by the elements $t_i t_{i+1}$, $i=1,\ldots,q-1$.

Now Theorem 4.2 of \cite{E10} can be extended to positive characteristic and recast as follows:

\begin{theorem}\label{A_fine}
Let $\FF$ be an algebraically closed field, $\chr{\FF}\neq 2$. Let $n\ge 3$ if $\chr{\FF}\ne 3$ and $n\ge 4$ if $\chr{\FF}=3$. Then any fine grading on $\Psl_n(\FF)$ is equivalent to one of the following:
\begin{itemize}
\item $\AI(T,k)$ as in Definition \ref{type_I_fine} with $k\sqrt{|T|}=n$,
\item $\AII(T,q,s,\tau)$ as in Definition \ref{type_II_fine} with $(q+2s)\sqrt{|T|}=n$.
\end{itemize}
Gradings belonging to different types listed above are not equivalent. Within each type, we have the following:
\begin{itemize}
\item $\AI(T_1,k_1)$ and $\AI(T_2,k_2)$ are equivalent if and only if\\ $T_1\cong T_2$ and $k_1=k_2$;
\item $\AII(T_1,q_1,s_1,\tau_1)$ and $\AII{}(T_2,q_2,s_2,\tau_2)$ are equivalent if and only if\\ $T_1\cong T_2$, $q_1=q_2$, $s_1=s_2$ and, identifying $T_1=T_2=\ZZ_2^{2r}$, $\Sigma(\tau_1)$ is conjugate to $\Sigma(\tau_2)$ by the natural action of $\mathrm{ASp}_{2r}(2)$.\qed
\end{itemize}
\end{theorem}

The missing case $n=\chr{\FF}=3$ can be treated using octonions, because in characteristic $3$ the algebra of traceless octonions under commutator is a Lie algebra isomorphic to $\Psl_3(\FF)$ (cf. \cite[Remark 4.11]{BK10}).

\subsection{Weyl groups of fine gradings}

By \cite[Theorem 2.8]{EK_Weyl}, the Weyl group of $\M(T,k)$ is isomorphic to $T^{k-1}\rtimes(\sg(k)\times\Aut(T,\beta))$, with $\sg(k)$ and $\Aut(T,\beta)$ acting on $T^{k-1}$ through their natural action on $T^k$ and identification of $T^{k-1}$ with $T^k/T$ where $T$ is imbedded into $T^k$ diagonally. Thanks to the isomorphism $\AAut(M_2(\FF))\to\AAut(\Sl_2(\FF))$, it follows that the Weyl group of the Cartan grading on $\Sl_2(\FF)$ is $\sg(2)$ (the classical Weyl group of type $A_1$) and the Weyl group of the Pauli grading on $\Sl_2(\FF)$ is $\SP_2(2)=\GL_2(2)$ (this is known in the case $\chr{\FF}=0$ --- see \cite{HPPT}).

To state our result for $\Psl_n(\FF)$, $n\ge 3$, it is convenient to introduce the following notation:
\[
\overline{\Aut}(T,\beta)\bydef\Aut(T,\beta)\rtimes\langle\sigma\rangle,
\]
where $\sigma$ is an element of order $2$ acting as the automorphism of $T$ that sends $a_i$ to $a_i^{-1}$ and $b_i$ to $b_i$, where $a_i$ and $b_i$ are the generators of $T$ used for the chosen realization of $\cD$ (a ``symplectic basis'' of $T$ with respect to $\beta$). We observe that $\beta(\sigma\cdot u,\sigma\cdot v)=\beta(u,v)^{-1}$, for all $u,v\in T$, and hence we obtain an induced action of $\sigma$ on $\Aut(T,\beta)$ by setting $(\sigma\cdot\alpha)(t)\bydef\sigma\cdot\alpha(\sigma\cdot t)$ for all $\alpha\in\Aut(T,\beta)$ and $t\in T$. The elements of $\overline{\Aut}(T,\beta)$ act as automorphisms of $T$ that send $\beta$ to $\beta^{\pm 1}$. However, this action is not faithful if $T$ is an elementary $2$-group.

\begin{theorem}\label{th:Weyl_AI}
Let $\FF$ be an algebraically closed field, $\chr{\FF}\neq 2$. Let $n\ge 3$ if $\chr{\FF}\ne 3$ and $n\ge 4$ if $\chr{\FF}=3$. Consider the fine grading $\Gamma=\AI(T,k)$ on $\Psl_n(\FF)$ as in Definition \ref{type_I_fine}, $k\sqrt{|T|}=n$. Then
\[
\W(\Gamma)\cong T^{k-1}\rtimes(\sg(k)\times\overline{\Aut}(T,\beta)),
\]
with $\sg(k)$ and $\overline{\Aut}(T,\beta)$ acting on $T^{k-1}$ through their natural action on $T^k$ and identification of $T^{k-1}$ with $T^k/T$ where $T$ is imbedded into $T^k$ diagonally.
\end{theorem}

\begin{proof}
The grading $\Gamma$ on $\cL=\Psl_n(\FF)$ is induced by the grading $\Gamma'=\M(T,k)$ on $\cR=M_n(\FF)$. The universal group of both gradings is $G=\wt{G}(T,k)^0$. Since restriction is a bijection between gradings on $\cR$ and Type I gradings on $\cL$, an automorphism $\psi'$ of $\cR$ sends ${}^\alpha\Gamma'$ to $\Gamma'$, for some automorphism $\alpha$ of $G$, if and only if the induced automorphism $\psi$ of $\cL$ sends ${}^\alpha\Gamma$ to $\Gamma$. The automorphism group of $\cL$ is the semidirect product of $\Aut(\cR)$, in its induced action on $\cL$, and $\langle\sigma\rangle$, where $\sigma$ is given by the negative of matrix transpose. To compute the action of $\sigma$, recall that $(u_1,\ldots,u_k)T\in T^k/T$ can be represented by the automorphism $X\mapsto DXD^{-1}$ where $D=\diag(X_{u_1},\ldots,X_{u_k})$, $\pi\in\sg(k)$ can be represented by $X\mapsto PXP^{-1}$ where $P$ is the permutation matrix corresponding to $\pi$, and $\alpha\in\Aut(T,\beta)$ can be represented by $X\mapsto\psi_0(X)$ where $\psi_0$ is an automorphism of $\cD$ such that $\psi_0(X_t)\in\FF X_{\alpha(t)}$ for all $t\in T$. The conjugation by $\sigma$ sends the automorphism $X\mapsto\Psi X\Psi^{-1}$ to the automorphism $X\mapsto({}^t\Psi^{-1})X({}^t\Psi)$, i.e., replaces $\Psi$ by ${}^t\Psi^{-1}$. Hence, $\sigma$ commutes with $\sg(k)$, while the conjugation by $\sigma$ sends $(u_1,\ldots,u_k)T$ to $(\sigma\cdot u_1,\ldots,\sigma\cdot u_k)T$, where the action of $\sigma$ on $T$ is as indicated above. Also, the action of $\sigma$ on $G$ sends $z_{i,j,t}\bydef \wt{g}_i t\wt{g}_j^{-1}$ to $z_{i,j,\sigma\cdot t}^{-1}$, so $\sigma$ belongs to $\Aut(\Gamma)$, but not to $\Stab(\Gamma)$. Hence we obtain $\Aut(\Gamma)=\Aut(\Gamma')\rtimes\langle\sigma\rangle$ and $\Stab(\Gamma)=\Stab(\Gamma')$. The result follows.
\end{proof}


\begin{theorem}\label{th:Weyl_AII}
Let $\FF$ be an algebraically closed field, $\chr{\FF}\neq 2$. Let $n\ge 3$ if $\chr{\FF}\ne 3$ and $n\ge 4$ if $\chr{\FF}=3$. Consider the fine grading $\Gamma=\AII(T,q,s,\tau)$ on $\Psl_n(\FF)$ as in Definition \ref{type_II_fine}, $(q+2s)\sqrt{|T|}=n$. Let $\Sigma=\Sigma(\tau)$. Then $\W(\Gamma)$ contains a normal subgroup $N$ isomorphic to $\ZZ_2^{q+s-1}$ such that
\[
\W(\Gamma)/N\cong \big((T^{q+s-1}\times\ZZ_2^s)\rtimes(\sg\Sigma\times\sg(s)\big)\rtimes \Aut^*\Sigma,
\]
where the actions are described naturally if we identify $T^{q+s-1}$ with $T^{q+s}/T$ and $\ZZ_2^{q+s-1}$ with $\ZZ_2^{q+s}/\ZZ_2$ (diagonal imbeddings).
Moreover, $\W(\Gamma)$ contains a subgroup isomorphic to $\big((T^{q+s-1}\times\ZZ_2^s)\rtimes(\sg\Sigma\times\sg(s)\big)\rtimes \Aut\Sigma$ that is disjoint from $N$.
\end{theorem}

\begin{proof}
The grading $\Gamma=\AII(T,q,s,\tau)$ on $\cL=\Psl_n(\FF)$ is induced by the grading $\Gamma'$ on $\brac{\cR}$, where $\cR=M_n(\FF)$, obtained from $\wb{\Gamma}'=\M(T,q,s,\tau)$ and $\vphi$ as in Definition \ref{type_II_fine}. The universal group of $\wb{\Gamma}'$ is $\bG=\wt{G}(T,q,s,\tau)^0$, while the universal group of $\Gamma$ is the extension $G$ of $\bG$ as in Lemma \ref{lm:extend_group}.
Similarly to Type I, an automorphism $\psi'$ of $\cR$ sends ${}^\alpha\Gamma'$ to $\Gamma'$, for some automorphism $\alpha$ of $G$, if and only if the induced automorphism $\psi$ of $\cL$ sends ${}^\alpha\Gamma$ to $\Gamma$. Note that $\alpha$ fixes the distinguished element $h=(\bar{e},-1)$ and hence yields an automorphism $\wb{\alpha}$ of $\bG$. It follows that $\psi'$ sends ${}^{\wb{\alpha}}\wb{\Gamma}'$ to $\wb{\Gamma}'$. For any $g\in G$ and $X\in\cR_g$, we have $\vphi(X)=-\chi(g)X$. Since $(\psi')^{-1}(X)\in\cR_{\alpha^{-1}(g)}$, we also have $(\vphi(\psi')^{-1})(X)=-\chi(\alpha^{-1}(g))(\psi')^{-1}(X)$. It follows that $\psi'\vphi(\psi')^{-1}=\xi\vphi$ where $\xi$ is the action of the character $(\chi\circ\alpha^{-1})\chi^{-1}$ on $\cR$ determined by the $G$-grading $\Gamma'$. Since $\alpha(h)=h$, $(\chi\circ\alpha^{-1})\chi^{-1}$ can be regarded as a character of $\bG$, hence $\xi$ belongs to  $\Diag(\wb{\Gamma}')$. Conversely, if $\psi'$ sends ${}^{\wb{\alpha}}\wb{\Gamma}'$ to $\wb{\Gamma}'$ and $\psi'\vphi(\psi')^{-1}=\xi\vphi$ for some $\xi\in\Diag(\wb{\Gamma}')$, then for any $\bg\in\bG$ and $X\in\cR_\bg$, we have $\psi'(X)\in\cR_{\wb{\alpha}(\bg)}$ and $\vphi(\psi'(X))=\nu\psi'(X)$ where $\nu\in\FF^\times$ depends only on $\bg$. It follows that $\psi'$ permutes the components of $\Gamma'$ and hence sends ${}^\alpha\Gamma'$ to $\Gamma'$ where $\alpha$ is a lifting of $\wb{\alpha}$. We have proved that an automorphism $\psi'$ of $\cR$ belongs to $\Aut^*(\wb{\Gamma}',\vphi)$, respectively $\Stab(\wb{\Gamma}',\vphi)$, if and only if the induced automorphism $\psi$ of $\cL$ belongs to $\Aut(\Gamma)$, respectively $\Stab(\Gamma)$. Finally, note that $-\vphi$ induces an automorphism of $\cL$ that belongs to $\Stab(\Gamma)$. It follows that the Weyl group of $\Gamma$ is isomorphic to $\Aut^*(\wb{\Gamma}',\vphi)/\Stab(\wb{\Gamma}',\vphi)$. The latter group was described in Theorem \ref{groups_fine_phi_grad_matrix}.
\end{proof}

If $\chr{\FF}=3$, there are two fine gradings on $\Psl_3(\FF)$: the Cartan grading, whose universal group is $\ZZ^2$, and the grading induced by the Cayley--Dickson doubling process for octonions, whose universal group is $\ZZ_2^3$. The Weyl groups of these gradings are, respectively, the classical Weyl group of type $G_2$ \cite[Theorem 3.3]{EK_Weyl} and $\GL_3(2)$ \cite[Theorem 3.5]{EK_Weyl}.


\section{Series $B$, $C$ and $D$}\label{se:BCD}

In this section we describe the Weyl groups of fine gradings on the simple Lie algebras of series $B$, $C$ and $D$ with exception of type $D_4$. Thus, we take $\cR=M_n(\FF)$, $n\ge 4$, and $\cL=\sks(\cR,\vphi)$ where $\vphi$ is an involution on $\cR$. If $\vphi$ is symplectic, then, of course, $n$ has to be even. If $\vphi$ is orthogonal, we assume $n\ge 5$ and $n\ne 8$. First we review the classification of fine gradings on $\cL$ from \cite{E10} (extended to positive characteristic using automorphism group schemes) and then derive the Weyl groups for $\cL$ from what we already know about automorphisms of fine $\vphi$-gradings (Section \ref{se:matrix}) on $\cR$.

\subsection{Classification of fine gradings}

Under the stated assumptions on $n$, the restriction from $\cR$ to $\cL$ yields an isomorphism $\AAut(\cR,\vphi)\to\AAut(\cL)$ (see \cite[\S 3]{BK10}). It follows that the classification of fine gradings on $\cL$ is the same as the classification of fine $\vphi$-gradings on $\cR$ (here $\vphi$ is fixed).

The case of series $B$ is quite easy, because $n$ is odd and hence the elementary $2$-group $T$ must be trivial. Let $G=\wt{G}(\{e\},q,s,\tau)^0$ where $\tau=(e,\ldots,e)$, so $G\cong\ZZ_2^{q-1}\times\ZZ^s$.

\begin{df}\label{grad_B_fine}
Consider the grading $\Gamma=\M(\{e\},q,s,\tau)$ on $\cR$ by $G$. Let $\Phi$ be the matrix given by
\begin{equation*}
\Phi=\diag\left(\underbrace{1,\ldots,1}_q,\begin{bmatrix}0&1\\1&0\end{bmatrix},\ldots,\begin{bmatrix}0&1\\1&0\end{bmatrix}\right).
\end{equation*}
Then $\Gamma$ is a fine $\vphi$-grading for $\vphi(X)=\Phi^{-1}({}^tX)\Phi$ and hence its restriction is a fine grading on $\cL\cong\So_n(\FF)$. We will denote this grading by $\B(q,s)$.
\end{df}

Now we turn to series $C$ and $D$, where $n$ is even and hence $T$ may be nontrivial. So, let $T$ be an elementary $2$-group of even dimension. Choose $\tau$ as in \eqref{datum_2_tau_fine} with all $t_i\in T_-$ in case of series $C$ and all $t_i\in T_+$ in case of series $D$. Let $G=\wt{G}(T,q,s,\tau)^0$, so $G\cong\ZZ_2^{\dim T-2\dim T_0+\max(0,q-1)}\times\ZZ_4^{\dim T_0}\times\ZZ^s$
where $T_0$ is the subgroup of $T$ generated by the elements $t_i t_{i+1}$, $i=1,\ldots,q-1$.

\begin{df}\label{grad_CD_fine}
Consider the grading $\Gamma=\M(\cD,q,s,\tau)$ on $\cR$ by $G$ where $t_1\ne t_2$ if $q=2$ and $s=0$. Let $\Phi$ be the matrix given by
\begin{equation*}
\Phi=
\diag\left(X_{t_1},\ldots,X_{t_q},\begin{bmatrix}0&I\\\delta I&0\end{bmatrix},\ldots,\begin{bmatrix}0&I\\\delta I&0\end{bmatrix}\right),
\end{equation*}
where $\delta=-1$ for series $C$ and $\delta=1$ for series $D$. Then $\Gamma$ is a fine $\vphi$-grading for $\vphi(X)=\Phi^{-1}({}^tX)\Phi$ and hence its restriction is a fine grading on $\cL\cong\Sp_n(\FF)$ or $\So_n(\FF)$. We will denote this grading by $\C(T,q,s,\tau)$ or $\D(T,q,s,\tau)$, respectively.
\end{df}

The following three results are Theorem 5.2 of \cite{E10}, stated separately for series $B$, $C$ and $D$ (and extended to positive characteristic).

\begin{theorem}\label{B_fine}
Let $\FF$ be an algebraically closed field, $\chr{\FF}\neq 2$. Let $n\ge 5$ be odd. Then any fine grading on $\So_n(\FF)$ is equivalent to $\B(q,s)$ where $q+2s=n$. Also, $\B(q_1,s_1)$ and $\B(q_2,s_2)$ are equivalent if and only if $q_1=q_2$ and  $s_1=s_2$.\qed
\end{theorem}

\begin{theorem}\label{C_fine}
Let $\FF$ be an algebraically closed field, $\chr{\FF}\neq 2$. Let $n\ge 4$ be even. Then any fine grading on $\Sp_n(\FF)$ is equivalent to $\C(T,q,s,\tau)$ where $(q+2s)\sqrt{|T|}=n$. Moreover, $\C(T_1,q_1,s_1,\tau_1)$ and $\C(T_2,q_2,s_2,\tau_2)$ are equivalent if and only if $T_1\cong T_2$, $q_1=q_2$, $s_1=s_2$ and, identifying $T_1=T_2=\ZZ_2^{2r}$, $\Sigma(\tau_1)$ is conjugate to $\Sigma(\tau_2)$ by the twisted action of $\SP_{2r}(2)$ as in Definition \ref{df:twisted_action}.\qed
\end{theorem}

\begin{theorem}\label{D_fine}
Let $\FF$ be an algebraically closed field, $\chr{\FF}\neq 2$. Let $n\ge 6$ be even. Assume $n\ne 8$. Then any fine grading on $\So_n(\FF)$ is equivalent to $\D(T,q,s,\tau)$ where $(q+2s)\sqrt{|T|}=n$. Moreover, $\D(T_1,q_1,s_1,\tau_1)$ and $\D(T_2,q_2,s_2,\tau_2)$ are equivalent if and only if $T_1\cong T_2$, $q_1=q_2$, $s_1=s_2$ and, identifying $T_1=T_2=\ZZ_2^{2r}$, $\Sigma(\tau_1)$ is conjugate to $\Sigma(\tau_2)$ by the twisted action of $\SP_{2r}(2)$ as in Definition \ref{df:twisted_action}.\qed
\end{theorem}

\subsection{Weyl groups of fine gradings}

Let $\Gamma=\B(q,s)$, $\C(T,q,s,\tau)$ or $\D(T,q,s,\tau)$, so $\Gamma$ is the restriction of the grading $\Gamma'=\M(T,q,s,\tau)$ on $\cR$ to $\cL=\sks(\cR,\vphi)$. By arguments similar to the proof of Theorem \ref{th:Weyl_AII}, one shows that the Weyl group of $\Gamma$ is isomorphic to $\Aut(\Gamma',\vphi)/\Stab(\Gamma',\vphi)$, which was described in Theorem \ref{groups_fine_phi_grad_matrix}. For $\Gamma=\B(q,s)$, $T$ is trivial and $\Sigma$ is a singleton of multiplicity $q$, so we obtain:

\begin{theorem}\label{th:Weyl_B}
Let $\FF$ be an algebraically closed field, $\chr{\FF}\neq 2$. Let $n\ge 5$ be odd. Consider the fine grading $\Gamma=\B(q,s)$ on $\So_n(\FF)$ as in Definition \ref{grad_B_fine}, where $q+2s=n$. Let $\Sigma=\Sigma(\tau)$. Then
$\W(\Gamma)\cong\sg(q)\times W(s)$ where $W(s)=\ZZ_2^s\rtimes\sg(s)$ (wreath product of $\sg(s)$ and $\ZZ_2$).\qed
\end{theorem}

For $\C(T,q,s,\tau)$ and $\D(T,q,s,\tau)$, $T$ may be nontrivial, so the answer is more complicated:

\begin{theorem}\label{th:Weyl_CD}
Let $\FF$ be an algebraically closed field, $\chr{\FF}\neq 2$. Let $n\ge 4$ be even. Consider the fine grading $\Gamma=\C(T,q,s,\tau)$ on $\Sp_n(\FF)$ or $\Gamma=\D(T,q,s,\tau)$ on $\So_n(\FF)$ as in Definition \ref{grad_CD_fine}, where $(q+2s)\sqrt{|T|}=n$ and $n\ne 4,8$ in the case of $\So_n(\FF)$. Let $\Sigma=\Sigma(\tau)$. Then
\[
\W(\Gamma)\cong\big((T^{q+s-1}\times\ZZ_2^s)\rtimes(\sg\Sigma\times\sg(s)\big)\rtimes \Aut\Sigma,
\]
where the actions on $T^{q+s-1}$ are via the identification with $T^{q+s}/T$ (diagonal imbedding).
\end{theorem}



\end{document}